\documentclass[letterpaper,11pt]{amsart}
\usepackage[margin=1.2in]{geometry}
\usepackage{amsmath,amsthm,amssymb}
\usepackage{xspace,xcolor}
\usepackage[breaklinks,colorlinks,citecolor=teal,linkcolor=teal,urlcolor=teal,pagebackref,hyperindex]{hyperref}
\usepackage[alphabetic]{amsrefs}
\usepackage[all]{xy}
\usepackage[english]{babel}
\usepackage{enumitem}
\usepackage{tikz,xcolor}
\usepackage{tikz-cd}
\usepackage{mathrsfs}
\usepackage{color}

\newcounter{intro}

\newtheorem{intro-conjecture}[intro]{Conjecture}
\newtheorem{intro-corollary}[intro]{Corollary}
\newtheorem{intro-theorem}[intro]{Theorem}

\newcommand{\theoremref}[1]{\hyperref[#1]{Theorem~\ref*{#1}}}
\newcommand{\lemmaref}[1]{\hyperref[#1]{Lemma~\ref*{#1}}}
\newcommand{\definitionref}[1]{\hyperref[#1]{Definition~\ref*{#1}}}
\newcommand{\propositionref}[1]{\hyperref[#1]{Proposition~\ref*{#1}}}
\newcommand{\conjectureref}[1]{\hyperref[#1]{Conjecture~\ref*{#1}}}
\newcommand{\corollaryref}[1]{\hyperref[#1]{Corollary~\ref*{#1}}}
\newcommand{\exampleref}[1]{\hyperref[#1]{Example~\ref*{#1}}}

\theoremstyle{plain}
\newtheorem{thm}{Theorem}[section]

\newtheorem{lem}[thm]{Lemma}
\newtheorem{prop}[thm]{Proposition}
\newtheorem{cor}[thm]{Corollary}

\theoremstyle{definition}
\newtheorem{defi}[thm]{Definition}

\newtheorem{eg}[thm]{Example}

\theoremstyle{remark}
\newtheorem{rmk}[thm]{Remark}

\def\Mustata{Mus\-ta\-\c{t}\u{a}\xspace}

\def\N{{\mathbf N}}
\def\Z{{\mathbf Z}}
\def\Q{{\mathbf Q}}
\def\R{{\mathbf R}}
\def\C{{\mathbf C}}

\def\A{{\mathbf A}}

\def\cA{\mathcal{A}}
\def\cB{\mathcal{B}}

\def\cD{\mathcal{D}}

\def\cH{\mathcal{H}}

\def\cK{\mathcal{K}}

\def\cM{\mathcal{M}}
\def\cN{\mathcal{N}}
\def\cO{\mathcal{O}}

\def\cQ{\mathcal{Q}}

\def\cU{\mathcal{U}}

\def\tcM{\widetilde{\mathcal{M}}}

\def\O{\mathcal{O}}

\def\.{\cdot}
\def\^{\widehat}

\def\de{\partial}

\def\({\left(}
\def\){\right)}

\newcommand{\cx}[1]{\left[{{}#1}\right]}

\renewcommand{\and}{ \ \ \text{ and } \ \ }

\begin{document}

\author[Q.~Chen]{Qianyu Chen}

\address{Department of Mathematics, University of Michigan, 530 Church Street, Ann Arbor, MI 48109, USA}

\email{qyc@umich.edu}

\author[B.~Dirks]{Bradley Dirks}

\address{Department of Mathematics, Stony Brook University, Stony Brook, NY 11794-3651, USA}

\email{bradley.dirks@stonybrook.edu}

\author[S.~Olano]{Sebasti\'{a}n Olano}

\address{Department of Mathematics, University of Toronto, 40 St. George St., Toronto, Ontario Canada, M5S 2E4}

\email{seolano@math.toronto.edu}

\thanks{ QC was partially supported by NSF grant DMS-1952399. B.D. was partially supported by the National Science Foundation under Grant No. DMS-1926686 and MSPRF DMS-2303070.}

\title{Restrictions of mixed Hodge modules using generalized $V$-filtrations}

\begin{abstract} We study generalized $V$-filtrations, defined by Sabbah, on $\mathcal D$-modules underlying mixed Hodge modules on $X\times \mathbf A^r$. Using cyclic covers, we compare these filtrations to the usual $V$-filtration, which is better understood. The main result shows that these filtrations can be used to compute the restriction functors $\sigma^!, \sigma^*$, where $\sigma \colon X \times \{0\} \to X \times \mathbf A^r$ is the inclusion of the zero section.

As an application, we use the restriction result to study singularities of complete intersection subvarieties. These filtrations can be used to study the local cohomology mixed Hodge module. In particular, we classify when weighted homogeneous isolated complete intersection singularities in $\mathbf A^n$ are $k$-Du Bois and $k$-rational.
\end{abstract}

\maketitle

\section{Introduction} Over the complex numbers, singularities of local complete intersection subvarieties have recently been studied using Saito's theory of mixed Hodge modules \cites{MustataPopaDuBois,CDMO,CDM1}. One of the cornerstones of these applications is the extension, beyond the hypersurface case, of the relation between the Hodge module structure on local cohomology with $V$-filtrations and Bernstein-Sato polynomials.

A key technical tool in extending this relationship is an understanding of the $V$-filtration of mixed Hodge modules along higher codimension smooth subvarieties. This $V$-filtration and its relation to Bernstein-Sato polynomials was first introduced in \cite{BMS}. The Hodge module theoretic properties were further studied in \cites{CD,CDS}. For a brief review of $V$-filtrations and mixed Hodge modules, see Section \ref{sect-preliminaries} below.

For a smooth complex algebraic variety $X$, we consider $T = X\times \A^r_t$ with coordinates $t_1,\dots, t_r$ on $\A^r_t$. Kashiwara and Malgrange showed that any $\cD_T$-module $\cM$ underlying a mixed Hodge module $M$ on $T$ admits a $V$-filtration $(V^\lambda \cM)_{\lambda \in \Q}$ by $\cD_X$ (not $\cD_T$)-submodules.

Three important properties of this filtration are that it is discretely indexed, we have
\[ t_i V^\lambda \cM \subseteq V^{\lambda+1}\cM, \quad \de_{t_i} V^{\lambda}\cM \subseteq V^{\lambda-1}\cM,\]
and the shifted Euler operator $\sum_{i=1}^r t_i \de_{t_i} - \lambda + r$ is nilpotent on ${\rm Gr}_V^\lambda(\cM)$.

We can define Koszul-like complexes
\[A^\lambda(\cM) = \left[V^\lambda \cM \xrightarrow[]{t_i} \bigoplus_{i=1}^r V^{\lambda+1}\cM \xrightarrow[]{t_i} \dots \xrightarrow[]{t_i} V^{\lambda+r}\cM\right],\]
and
\[B^\lambda(\cM) = \left[ {\rm Gr}_V^\lambda(\cM) \xrightarrow[]{t_i} \bigoplus_{i=1}^r {\rm Gr}_V^{\lambda+1}(\cM) \xrightarrow[]{t_i} \dots \xrightarrow[]{t_i} {\rm Gr}_V^{\lambda+r}(\cM) \right]
\]
placed in degree $0,1,\dots,r$. We also define
\[C^\lambda(\cM) = \left[ {\rm Gr}_V^{r+\lambda}(\cM) \xrightarrow[]{\de_{t_i}} \bigoplus_{i=1}^r {\rm Gr}_V^{r+\lambda-1}(\cM) \xrightarrow[]{\de_{t_i}} \dots \xrightarrow[]{\de_{t_i}} {\rm Gr}_V^{\lambda}(\cM) \right]\]
in degree $-r,\dots, 0$.

It is not hard to see (as shown in \propositionref{prop-KoszulVFilt} below) that most of these complexes are acyclic. The most important one (and essentially the only one) that is not necessarily acyclic is the complex for $\lambda = 0$. In fact, the cohomologies of this complex are shown to compute the $\cD$-module theoretic restriction of the module $\cM$ to the zero section $X \times \{0\} \subseteq X\times \A^r_t$.

The main results of \cites{CD,CDS} are extensions of the acyclicity and restriction results to the setting of mixed Hodge modules. Careful statements of these results are found in \theoremref{thm-CDResults} below.

The main objective of this paper is to extend these results to more general $V$-filtrations, defined by Sabbah \cite{Sabbah1}. The proper definition is given in Section \ref{sect-preliminaries} below. For now, we just mention that associated to a non-zero tuple $(a_1,a_2,\dots, a_r) \in \Z_{\geq 0}$, we can define a linear form $L = \sum_{i=1}^r a_i s_i$, which is called a \emph{slope}. Such a slope is \emph{non-degenerate} if $a_i > 0$ for all $i$. We set $|L| = \sum_{i=1}^r a_i$ and $L(t\de_t) = \sum_{i=1}^r a_i t_i \de_{t_i}$.

Sabbah defines the unique ${}^L V$-filtration on a $\cD_T$-module $\cM$ by similar properties to the $V$-filtration described above, except one requires $L(t\de_t) - \lambda + |L|$ to be nilpotent on ${\rm Gr}_{{}^L V}^{\lambda}(\cM)$, which for the remainder of the paper we will write as ${\rm Gr}_L^{\lambda}(\cM)$. This imposes the other conditions $t_i {}^L V^{\lambda}\cM \subseteq {}^L V^{\lambda+a_i}\cM$ and $\de_{t_i} {}^L V^{\lambda}\cM \subseteq {}^L V^{\lambda-a_i}\cM$. This filtration heavily depends on the ordered choice of coordinates $t_1,\dots, t_r$: for example, even reordering the coordinates gives a different filtration, corresponding to the slope with permuted coefficients.

If $L = \sum_{i=1}^r s_i$, then ${}^L V^\bullet \cM = V^\bullet \cM$ from above. In fact, if $L = \sum_{i\in I} s_i$ for some $I \subseteq \{1,\dots, r\}$, the filtration ${}^L V^\bullet \cM$ is the $V$-filtration along $\{t_i \mid i\in I\}$.

In analogy with the above, we can define Koszul-like complexes 
\[A^\lambda_L(\cM) = \left[{}^L V^\lambda \cM \xrightarrow[]{t_i} \bigoplus_{i=1}^r {}^L V^{\lambda+a_i}\cM \xrightarrow[]{t_i} \dots \xrightarrow[]{t_i} {}^L V^{\lambda+|L|}\cM\right],\]
\[B^\lambda_L(\cM) = \left[{\rm Gr}_L^\lambda(\cM) \xrightarrow[]{t_i} \bigoplus_{i=1}^r {\rm Gr}_L^{\lambda+a_i}(\cM) \xrightarrow[]{t_i} \dots \xrightarrow[]{t_i} {\rm Gr}_L^{\lambda+|L|}(\cM)\right]\]
placed in degree $0,1,\dots r$ and
\[C^\lambda_L(\cM) = \left[{\rm Gr}_L^{\lambda+|L|}(\cM) \xrightarrow[]{\de_{t_i}} \bigoplus_{i=1}^r {\rm Gr}_L^{\lambda+|L|-a_i}(\cM) \xrightarrow[]{\de_{t_i}} \dots \xrightarrow[]{\de_{t_i}} {\rm Gr}_L^{\lambda}(\cM)\right]\]
in degree $-r,\dots, 0$.

\propositionref{lem-DModuleRestriction} below shows that, at the $\cD$-module level, we have the acyclicity results for these complexes and the restriction result for $B^0_L(\cM)$. The main result of this note is the extension to the mixed Hodge module setting.

If $(\cM,F,W)$ is a bi-filtered left $\cD_T$-module underlying a mixed Hodge module $M$ on $T$, we define filtrations
\[F_p A^\lambda_L(\cM) = \left[ F_{p+r}{}^L V^\lambda \cM \xrightarrow[]{t_i} \bigoplus_{i=1}^r F_{p+r}{}^L V^{\lambda+a_i}\cM \xrightarrow[]{t_i} \dots \xrightarrow[]{t_i} F_{p+r}{}^L V^{\lambda+|L|}\cM\right],\]
\[F_p B^\lambda_L(\cM) = \left[F_{p+r}{\rm Gr}_L^\lambda(\cM) \xrightarrow[]{t_i} \bigoplus_{i=1}^r F_{p+r}{\rm Gr}_L^{\lambda+a_i}(\cM) \xrightarrow[]{t_i} \dots \xrightarrow[]{t_i} F_{p+r}{\rm Gr}_L^{\lambda+|L|}(\cM)\right],\]
\[F_p C^\lambda_L(\cM) = \left[F_p{\rm Gr}_L^{\lambda+|L|}(\cM) \xrightarrow[]{\de_{t_i}} \bigoplus_{i=1}^r F_{p+1}{\rm Gr}_L^{\lambda+|L|-a_i}(\cM) \xrightarrow[]{\de_{t_i}} \dots \xrightarrow[]{\de_{t_i}} F_{p+r}{\rm Gr}_L^{\lambda}(\cM)\right]\]

Moreover, the nilpotent operator $L(t\de_t) - \lambda +|L|$ on ${\rm Gr}_L^{\lambda}(\cM)$ and the induced filtration $M_\bullet {\rm Gr}_L^{\lambda}(\cM) = {\rm Gr}_L^{\lambda}(W_\bullet \cM)$ give rise to the \emph{relative monodromy filtration} $W_\bullet {\rm Gr}_L^{\lambda}(\cM)$. This is recalled in more detail in Section \ref{sect-preliminaries} below. The complexes $B^0_L(\cM)$ and $C^0_L(\cM)$ admit filtrations $W_\bullet B^0_L(\cM), W_\bullet C^0_L(\cM)$ by taking the relative monodromy filtration on each piece.

With this notation in place, we have the following:

\begin{intro-theorem} \label{thm-main} Let $L = \sum_{i=1}^r a_i s_i$ be a non-degenerate slope and let $(\cM,F,W)$ be a bi-filtered $\cD_T$-module underlying a mixed Hodge module. The complexes
\[ A_L^\chi(\cM,F) \text{ and } B_L^\chi(\cM,F) \text{ are filtered acyclic for all } \chi > 0,\]
and
\[ C_L^\chi(\cM,F) \text{ is filtered acyclic for all } \chi < 0.\]

Moreover, we have filtered quasi-isomorphisms
\[ \sigma^!(\cM,F) \cong A^0_L(\cM,F) \cong B^0_L(\cM,F),\]
and the latter two complexes are strictly filtered. Similarly,
\[ \sigma^*(\cM,F) \cong C^0_L(\cM,F),\]
and the rightmost complex is strictly filtered.

Finally, the filtrations $W_\bullet \cH^i B^0_L(\cM,F), W_\bullet \cH^i C_L^0(\cM,F)$ induced by $W_\bullet B^0_L(\cM,F), W_\bullet C^0_L(\cM,F)$ satisfy
\[ {\rm Gr}^W_k \cH^i B^0_L(\cM,F) \cong {\rm Gr}^W_{k+i} \cH^i \sigma^!(\cM,F),\]
\[ {\rm Gr}^W_k \cH^{-i} C^0_L(\cM,F) \cong {\rm Gr}^W_{k-i} \cH^{-i} \sigma^*(\cM,F),\]
as filtered $\cD_X$-modules underlying polarizable Hodge modules.
\end{intro-theorem}

\begin{rmk}
    The ideas in~\cite{CD} are not sufficient to prove \theoremref{thm-main} because the ${}^LV$-filtration depends on the choice of fiber coordinates.

    The ideas in~\cite{CDS} are used below; however, they do not automatically give us the strictness of the complexes $B^0_L(\cM,F)$ and $C^0_L(\cM,F)$. Indeed, the same problem arises in this situation: the ${}^L V$-filtration depends on the choice of coordinates $t_1,\dots, t_r$, and to prove strictness as in \cite{CDS}, one replaces $t_r$ by a general linear combination $\widetilde{t}_r$ of $t_1,\dots, t_r$, and considers the coordinates $t_1,\dots, \widetilde{t}_r$.
\end{rmk}

\begin{rmk} \label{rmk-closedProblems} Originally, we set out to find out whether one can understand the Fourier-Laplace transform of an $L$-monodromic mixed Hodge module. Ideally, such an understanding would have helped to relate the $B_L$ and $C_L$ complexes, hence proving the formula for $\sigma^*$ in \theoremref{thm-main} from the formula for $\sigma^!$. Indeed, by definition the Fourier-Laplace transform interchanges the action of $t_1,\dots, t_r$ (concerning the $!$-restriction) with $\partial_{y_1},\dots, \partial_{y_r}$ (concering the $*$-restriction), where $y_1,\dots, y_r$ are the corresponding dual coordinates on the dual vector bundle.

However, this is too optimistic: for $L \neq \sum_{i=1}^r s_i$, the Fourier-Laplace transform of an $L$-monodromic regular holonomic module $\cM$ is rarely regular holonomic. Indeed, by \cite{IT}*{Cor. 1.3}, if it were regular holonomic, then $\cM$ would have to be monodromic. So it is not possible for the Fourier-Laplace transform of an $L$-monodromic mixed Hodge module to remain a mixed Hodge module. One could instead look at irregular Hodge modules, but we do not pursue that in this paper.
\end{rmk}

The paper ends with an application of \theoremref{thm-main} to the study of singularities of local complete intersection subvarieties. If $Z \subseteq X$ is defined by a regular sequence $f_1,\dots, f_r$ in $\cO_X(X)$, we can define a pure Hodge module $B_f$ on $T$ of weight $n = \dim(X)$. Its underlying $\cD_T$-module is
\[ \cB_f = \bigoplus_{\alpha \in \N^r} \cO_X \de_t^\alpha \delta_f,\]
whose $\cD$-module structure is explained in Section \ref{sect-singularities} below. Here $\delta_f$ is a formal symbol: if one identifies $\cB_f \cong \cH^r_{\Gamma}(\cO_{X\times \A^r_t})$ (where $\Gamma = \Gamma(X)$ is the image of the graph embedding $\Gamma \colon X \to X\times \A^r_t$ along $f_1,\dots, f_r$), then $\delta_f$ corresponds to the class of $\frac{1}{(f_1-t_1)\dots (f_r-t_r)}$.

If $\sigma \colon X \times \{0\} \to T$ is the inclusion of the zero section, we have a natural isomorphism (by Saito's base change result \cite{SaitoMHM}*{(4.4.3)}, which says $\sigma^! \Gamma_* \Q_X^H[\dim X ] \cong i_* i^! \Q_X^H[\dim X]$, where $i\colon Z \hookrightarrow X$ is the closed embedding)
\[ \cH^r \sigma^! B_f \cong (\cH^r_Z(\cO_X),F,W),\]
where the right hand side is the bi-filtered local cohomology $\cD_X$-module along $Z$, which is traditionally computed using the \v{C}ech complex for $f_1,\dots ,f_r$.

By \cites{MustataPopaDuBois,CDMO,CDM1}, the Hodge module structure of $\cH^r_Z(\cO_X)$ is related to higher classes of singularities. These are the classes of $k$-Du Bois and $k$-rational singularities, where $k\in \Z_{\geq 0}$, whose definitions are reviewed in Section \ref{sect-singularities} below. For $k =0$, these agree with the classical notions of Du Bois and rational singularities, hence the terminology.

In \cite{CDMO}, the authors and \Mustata define a numerical invariant of $Z$, the \emph{minimal exponent} $\widetilde{\alpha}(Z)$,  in terms of $V^\bullet \cB_f$. The minimal exponent satisfies the following implications:
\[ \widetilde{\alpha}(Z) \geq r+k \iff Z \text{ has } k\text{-Du Bois singularities,}\]
\[ \widetilde{\alpha}(Z) > r+k \iff Z \text{ has } k\text{-rational singularities.}\]

From this, we see that $Z$ has $k$-rational singularities if it has $(k+1)$-Du Bois singularities, and it has $k$-Du Bois singularities if it has $k$-rational singularities. We say $Z$ has $k$-liminal singularities if $\widetilde{\alpha}(Z) = r+k$.

Our first main result in the study of singularities is the following description of the mixed Hodge module structure on local cohomology:
\begin{intro-theorem} \label{thm-LVDescribeLocalCohomology} Let $Z = V(f_1,\dots, f_r) \subseteq X$ be a complete intersection of pure codimension $r$. Let $L$ be a non-degenerate slope. Then for all $p, \ell \in \Z_{\geq 0}$, we have
\[ F_{p} W_{n+r+\ell} \cH^r_Z(\cO_X)\] \[ = \left\{ \sum_{|\alpha| \leq p} \frac{\alpha! h_\alpha}{f_1^{\alpha_1+1}\dots f_r^{\alpha_r+1}} \mid u = \sum_{|\alpha|\leq p} h_\alpha \de_t^\alpha \delta_f \in {}^L V^{|L|}\cB_f, L(t\de_t)^{\ell+1}u \in {}^L V^{>|L|}\cB_f\right\}.\]
\end{intro-theorem}

An immediate corollary of this computation is the following:

\begin{intro-corollary} \label{cor-LVMinExp} Let $Z = V(f_1,\dots, f_r) \subseteq X$ be a complete intersection of pure codimension $r$. Let $L$ be a non-degenerate slope. Then
\[ F_{k+r} \cB_f \subseteq {}^L V^{|L|}\cB_f \iff Z \text{ has }k\text{-Du Bois singularities},\]
\[ L(t\de_t) F_{k+r} \cB_f \subseteq {}^L V^{>|L|}\cB_f \iff Z \text{ has } k\text{-rational singularities.}\]
\end{intro-corollary}

Finally, the corollary allows us to prove the following, which is a generalization of the main result of \cite{CDM3}. Note that, here, we do not give an exact formula for the minimal exponent, but what we show is enough for the singularity classification.

Let $f_1,\dots, f_r \in \C[x_1,\dots ,x_n]$ be weighted homogeneous of (integer) degrees $d_1 \leq \dots \leq d_r$ with weights $(w_1,\dots, w_n) \in \Z^n_{\geq 1}$. Assume $Z = V(f_1,\dots, f_r) \subseteq \A^n_x$ is a complete intersection of pure codimension $r$ with an isolated singularity at $0$. 

\begin{rmk} The upper bound in the following corollary is a general result, due to \cite{CDM3}*{Thm. 1.2}.

The use of the ${}^L V$-filtration to establish the lower bound is essential. Indeed, if $d_1 = \dots = d_r = d$, then a stronger result is known by \cite{CDMO}*{Ex. 4.23}. However, the main point of that case is that one can explicitly compute the $V$-filtration using weighted homogeneity and Kashiwara's equivalence.

If the $d_i$ are not all equal, then it is more natural to consider the ${}^L V$-filtration with $L = \sum_{i=1}^r d_i s_i$, because this can be understood using weighted homogeneity and Kashiwara's equivalence (see \eqref{eq-homogeneity} below).
\end{rmk}

\begin{intro-corollary} \label{cor-ComputeMinExp}
The complete intersection $Z$ has Du Bois (hence, log canonical) singularities if and only if $|w| = \sum_{i=1}^n w_i \geq d_1 + \dots +d _r$. In this case, let
\[ k = \left \lfloor \frac{|w|-\sum_{i=1}^r d_i}{d_r} \right\rfloor.\]

Then
\[ r + k\leq \widetilde{\alpha}_0(Z) \leq   r + \frac{|w|-\sum_{i=1}^r d_i}{d_r},\]
and in fact, $\widetilde{\alpha}_0(Z) = r + k$ if and only if $d_r \mid (|w| - \sum_{i=1}^r d_i)$.

In particular, for this value of $k$, we see
\[ Z \text{ has } k\text{-liminal singularities near } 0 \iff d_r \mid \left(|w| - \sum_{i=1}^r d_i\right)\]
\[ Z \text{ has } k\text{-rational singularities near } 0 \iff d_r \nmid \left(|w| - \sum_{i=1}^r d_i\right).\]
\end{intro-corollary}

We expect that the equality $\widetilde{\alpha}_0(Z) =  r + \frac{|w|-\sum_{i=1}^r d_i}{d_r}$ holds in the setting of~\corollaryref{cor-ComputeMinExp}.

\textit{Outline.} Section \ref{sect-preliminaries} contains the background needed for the proofs of the main results. In Subsection \ref{subsect-MHM}, we provide a review of the theory of mixed Hodge modules, including a review of hypersurface $V$-filtrations. Subsection \ref{subsect-HigherCodimV} reviews the results for higher codimension $V$-filtrations. The definition and properties of the filtration ${}^L V^\bullet \cM$ are given in Section \ref{sect-LV}. The Verdier specialization process is used in Subsection \ref{subsect-Specializations} to study ${}^L V$-filtrations on mixed Hodge modules, using the properties of hypersurface $V$-filtrations.

Section \ref{sect-CyclicCovers} contains the proof of \theoremref{thm-main}. It begins with an analysis of mixed Hodge modules under cyclic coverings $X\times \A^r_w \to X\times \A^r_t$ defined by 
\[(x,w_1,\dots, w_r) \mapsto (x,w_1^{a_1},\dots, w_r^{a_r}).\] 

The main point is that the usual $V$-filtration on $\pi^!(\cM)$ can be related to the ${}^L V$-filtration on $\cM$. Two difficulties are that $\pi^!(-)$ need not preserve pure modules (though we give a criterion for when a pure Hodge module pulls back to a pure module, in a special case) and the Hodge filtration is not easy to understand. We work around this using an inductive strategy on the number of $a_i$ which are strictly larger than $1$. By focusing on this case (say only $a_i > 1$), we can leverage the fact that, over the \'{e}tale locus $\{w_i \neq 0\}$, the Hodge filtration on $\pi^!(\cM)$ is fully understood in terms of the Hodge filtration on $\cM$. Then we use the $V$-filtration along $w_i$ to get control on the Hodge filtration of $V^0\pi^!(\cM)$, which suffices to compute the restriction functor.

The final Section \ref{sect-singularities} contains the proofs of \theoremref{thm-LVDescribeLocalCohomology} and \corollaryref{cor-LVMinExp} and the example of weighted homogeneous complete intersections with isolated singularities. In particular, The proof of \corollaryref{cor-ComputeMinExp} is a combination of results in Section \ref{whci}. Though we cannot give an exact computation of the minimal exponent in the latter example (except in special cases), we give an easy criterion to check whether such a subvariety has $k$-Du Bois or $k$-rational singularities.

\noindent {\bf Acknowledgments.} The authors would like to thank Mircea \Mustata for his thoughtful advice. We would also like to thank Lei Wu for many helpful discussions on the topic of the paper. The second author would like to thank Christian Schnell and Claude Sabbah for many enlightening conversations about ideas occurring in this work, as well as Tong Zhou for pointing out the reference \cite{IT} used in Remark \ref{rmk-closedProblems}.

\section{Preliminaries} \label{sect-preliminaries}
We do not provide a review of the theory of $\cD$-modules in this paper, though we will review the necessary notions as they arise. For reference, see \cite{HTT}.

\subsection{Mixed Hodge modules and V-filtrations of hypersurfaces} \label{subsect-MHM} In this subsection, we discuss the relevant aspects of the theory of mixed Hodge modules. For details, one should consult Saito's papers \cites{saitoMHP,SaitoMHM} or Schnell's survey article \cite{SchnellMHM}.

Let $X$ be a smooth algebraic variety over $\C$ with $\dim(X) = n$. Saito associates to $X$ an abelian category ${\rm MHM}(X)$, the category of \emph{mixed Hodge modules} on $X$. A mixed Hodge module on $X$ consists of the following data: a filtered regular holonomic $\cD_X$-module $(\cM,F)$, with a finite filtration $W_\bullet \cM$ by sub-$\cD_X$-modules, a finite filtered $\Q$-perverse sheaf $(\cK,W)$, and a comparison morphism
\[ \alpha \colon (\cK, W) \otimes_{\Q} \C \to {\rm DR}_X(\cM,W)\]
which is a filtered quasi-isomorphism. These data are subject to various conditions which we will not fully explain here. Some of the conditions are the following, which concern the interaction between the Hodge filtration $F_\bullet \cM$ and the $V$-filtration $V^\bullet \cM$ on $\cM$ along any (locally defined) function $f \in \cO_X$.

First, we recall the definition of the $V$-filtration of $\cD_{X\times \A^1_t}$-modules along $t$, which is the coordinate on $\A^1$. Let $\cM$ be a coherent $\cD_{X\times \A^1_t}$-module.

\begin{defi} \label{defi-VFilt} The $\cD_{X\times \A^1_t}$-module $\cM$ admits a \emph{$V$-filtration along $t$}, which, following Saito \cite{saitoMHP}*{Sect. 3}, is an exhaustive decreasing, discrete (meaning there exists an increasing sequence $\alpha_j \in \Q$ with $\lim_{j\to -\infty} \alpha_j = -\infty$ and $\lim_{j\to \infty} \alpha_j = \infty$, such that $V^\chi\cM$ for $\chi \in (\alpha_j,\alpha_{j+1})$ only depends on $j$) and left continuous (meaning $V^\chi\cM = \bigcap_{\beta < \chi}V^\beta\cM$. In other words, $V^\chi\cM$ is constant for $\chi \in (\alpha_j,\alpha_{j+1}]$), $\Q$-indexed filtration $V^\bullet\cM$ satisfying the following conditions:
\begin{enumerate} \item(Coherence) $V^\alpha \cM$ is coherent over $V^0\cD_{X\times \A^1} = \cD_X[t,t\de_t]$,
\item(Compatibility, 1) \label{cond-good} $t V^\alpha \cM \subseteq V^{\alpha+1}\cM$ for all $\alpha \in \Q$, with equality for $\alpha \gg 0$,
\item(Compatibility, 2) $\de_t V^{\alpha}\cM \subseteq V^{\alpha-1}\cM$ for all $\alpha \in \Q$,
\item(Nilpotence) \label{cond-nilpotent} the operator $t\de_t - \alpha +1$ is nilpotent on ${\rm Gr}_V^{\alpha}(\cM) = V^\alpha\cM/V^{>\alpha}\cM$, where $V^{>\alpha}\cM = \bigcup_{\beta > \alpha} V^\beta \cM$.
\end{enumerate}
\end{defi}

\begin{rmk} Below, we will also use the operator \[s= -\de_{t} t\] which is more natural in the study of singularities and $b$-functions. We clearly have the equality $s = -(t\de_t+1)$ so the Nilpotence Condition \eqref{cond-nilpotent} of Definition \ref{defi-VFilt} can be restated as requiring $s+\alpha$ to be nilpotent on ${\rm Gr}_V^\alpha(\cM)$.
\end{rmk}

\begin{rmk} If a $V$-filtration on $\cM$ along $t$ exists, then it is unique. Hence, existence is an intrinsic property of the $\cD$-module $\cM$. 
\end{rmk}

\begin{eg} \label{eg-Kashiwara} Let $\cM$ be supported in $\{t=0\}$. Then Kashiwara's equivalence (see \cite{HTT}*{Sect. 1.6}) tells us that $\cM = i_+ \cM_0$ where $\cM_0 = \ker(t) \subseteq \cM$ is a $\cD_{X\times \{0\}}$-module and $i\colon X\times \{0\} \to X\times \A^1_t$ is the closed embedding. Thus,
\[ \cM =\bigoplus_{k\geq 0} \cM_0 \de_t^k,\]
and it is not hard to check that
\[ V^\lambda \cM = V^{\lceil \lambda \rceil} \cM = \bigoplus_{k\leq -\lceil \lambda \rceil} \cM_0 \de_t^k\]
is a $V$-filtration of $\cM$ along $t$.
\end{eg}

The uniqueness of $V$-filtrations implies the following:
\begin{lem} \label{lem-VFiltStrict} Let $0 \to \cM_1 \to \cM_2 \to \cM_3 \to 0$ be a short exact sequence of $\cD_{X\times \A^1_t}$-modules such that $\cM_i$ admits a $V$-filtration along $t$ for $i=1,2,3$. Then, for all $\lambda \in \Q$, the sequence
\[ 0 \to V^\lambda \cM_1 \to V^\lambda \cM_2 \to V^\lambda \cM_3 \to 0\]
is exact.
\end{lem}

From this and \exampleref{eg-Kashiwara}, we see that positive pieces of the $V$-filtration only depend on the restriction to $\{t\neq0\}$.

\begin{lem} \label{lem-PositivePieces} Let $\varphi \colon \cM \to \cN$ be a morphism of $\cD_{X\times \A^1_t}$-modules such that $\varphi\vert_{\{t\neq0\}}$ is an isomorphism. Then for all $\lambda > 0$, $\varphi$ induces an isomorphism
\[ \varphi \colon V^\lambda \cM \cong V^\lambda \cN.\]
\end{lem}

For $f\in \cO_X(X)$, let $i_f\colon X \to X\times \A^1_t$ be the graph embedding along $f$. We say that the $V$-filtration of a $\cD_X$-module $\cM$ along $f$ exists if the $V$-filtration of $i_{f,+}(\cM)$ along $t$ exists. For $\cD_X$-modules underlying a mixed Hodge module, the $V$-filtration exists for any locally defined function $f\in \cO_X$.

Returning to a $\cD_{X\times \A^1}$-module $\cM$, it is not hard to see using the Nilpotence Condition \eqref{cond-nilpotent} of Definition \ref{defi-VFilt} that the maps
\[ t \colon {\rm Gr}_V^\alpha(\cM) \to {\rm Gr}_V^{\alpha+1}(\cM),\]
\[ \de_t \colon {\rm Gr}_V^{\alpha+1}(\cM) \to {\rm Gr}_V^{\alpha}(\cM)\]
are isomorphisms for all $\alpha \neq 0$. In fact, the Compatibility Condition \eqref{cond-good} of Definition \ref{defi-VFilt} shows that $t \colon V^\alpha \cM \to V^{\alpha+1}\cM$ is an isomorphism for all $\alpha > 0$.

If $(\cM,F)$ is a filtered $\cD_X$-module, then $i_{f,+}(\cM)$ is also a filtered $\cD_{X\times \A^1_t}$-module. Indeed, we can write $i_{f,+}(\cM) = \bigoplus_{k\geq 0} \cM \de_t^k \delta_f$, and we have
\[ F_p i_{f,+}(\cM) = \bigoplus_{k\geq 0} F_{p-k-1}\cM \de_t^k \delta_f,\]
where the shift by $1$ is a normalizing convention due to the relative dimension of $i_f \colon X \to X\times \A^1_t$.

Saito imposes the following conditions on the filtration $F_\bullet \cM$: for any $f\in \cO_X(X)$, we have
\begin{equation} \label{cond-tFiltIso} t \colon F_p V^\alpha i_{f,+}(\cM) \to F_p V^{\alpha+1} i_{f,+}(\cM) \text{ is an isomorphism for all } \alpha > 0,\end{equation}
\begin{equation} \label{cond-dFiltIso} \de_t \colon F_p {\rm Gr}_V^{\alpha+1}(i_{f,+}(\cM)) \to F_{p+1} {\rm Gr}_V^{\alpha}(i_{f,+}(\cM)) \text{ is an isomorphism for all } \alpha < 0.\end{equation}

\begin{rmk} \label{rmk-goodnessOfFOnV0} Another property imposed on filtered  $\cD$-modules underlying mixed Hodge modules is the following: for $(\cM,F)$ underlying a mixed Hodge module on $X\times \A^1_t$, the filtration induced by $F_\bullet \cM$ on ${\rm Gr}_V^\lambda(\cM)$ is a good filtration.

In fact, we have the following (see \cite{saitoMHP}*{Cor. 3.4.7} and \cite{MHMProj}*{Prop. 10.6.5}): for any $\lambda \in \Q$, let $R_F(V^\lambda \cM) = \bigoplus_{p \in \Z} F_p V^\lambda \cM z^p$. This is a module over $R_F(V^0\cD_{X\times \A^1_t})$, and in fact is coherent over that ring. It is even coherent over the subring $R_F(\cD_{X\times \A^1_t/\A^1_t}) = R_F(\cD_X[t])$, due to the regular holonomicity.
\end{rmk}

For $i\colon H = V(f) \to X$ the inclusion of the hypersurface defined by $f$ into $X$, we can restrict a mixed Hodge module $M$ on $X$ along $i$ to get $i_* i^! M \in D^b({\rm MHM}(X))$. In fact, for underlying filtered $\cD$-modules, this restriction is given by the morphism
\[ t \colon {\rm Gr}_V^0(i_{f,+}(\cM,F[-1])) \to {\rm Gr}_V^1(i_{f,+}(\cM,F[-1])),\]
which is the motivation for \theoremref{thm-main}. Here the shift on the filtration is due to our use of left filtered $\cD_{X\times \A^1_t}$-modules.

In the remainder of this section, we will work with $M$ a mixed Hodge module on $X\times \A^1_t$.

The following lemma comes immediately from conditions \eqref{cond-tFiltIso} and \eqref{cond-dFiltIso}, respectively. See \cite{saitoMHP}*{Prop. 3.2.2, Rmk. 3.2.3} for the argument.

\begin{lem} Let $(\cM,F)$ underlie a mixed Hodge module on $X\times \A^1_t$. Let $j\colon \{t\neq 0\} \to X\times \A^1_t$. Then for all $\lambda > 0$ and $p\in \Z$, we have equality
\[ F_pV^{\lambda} \cM = V^{\lambda}\cM \cap j_*(F_p j^*\cM)\]
and
\[ F_p \cM = \sum_{k\geq 0} \de_t^k(F_{p-k}V^0\cM). \]
\end{lem}

A mixed Hodge module $M$ is \emph{pure of weight $d$} if ${\rm Gr}^W_i(M) = 0$ for all $i\neq d$. By definition, any pure Hodge module of weight $d$ decomposes into its \emph{strict support} decomposition. This means
\[ M = \bigoplus_{Z\subseteq X \times \A^1_t} M_Z\]
where $M_Z$ is a pure Hodge module of weight $d$ with strict support $Z \subseteq X \times \A^1_t$. Here $Z$ is an irreducible closed subset of $X \times \A^1_t$ and strict support means that the $\cD$-module underlying $M_Z$ admits no non-zero quotient or sub-object with support contained in a proper closed subset of $Z$. Equivalently, the underlying perverse sheaf is an intersection complex.

\begin{lem} \label{lem-strictSuppDMod} A $\cD_{X\times \A^1_t}$-module $\cM$ admits no non-zero sub-objects supported in $\{t=0\}$ if and only if the map
\[ t \colon {\rm Gr}_V^0(\cM) \to {\rm Gr}_V^1(\cM) \text{ is injective}.\]

It admits no non-zero quotient object supported in $\{t=0\}$ if and only if the map
\[ \de_t \colon {\rm Gr}_V^1(\cM) \to {\rm Gr}_V^0(\cM) \text{ is surjective,}\]
which holds if and only if $\cM = \cD_{X\times \A^1_t} \cdot V^{>0}\cM$.
\end{lem}

If $(\cM,F)$ underlies a Hodge module, then the morphisms in the previous lemma statement are automatically strict with respect to the Hodge filtration. This gives the following:

\begin{lem} \label{lem-strictSupportHodge} If $(\cM,F)$ underlies a Hodge module with no sub-object supported on $\{t=0\}$, then
\[ F_p V^0 \cM = V^0\cM \cap j_*(F_p j^*(\cM)).\]

If $(\cM,F)$ underlies a Hodge module with no quotient supported on $\{t=0\}$, then
\[ F_p\cM = \sum_{k \geq 0} \de_t^k(F_{p-k}V^{>0}\cM ).\]
\end{lem}

\subsection{Relative Monodromy Filtration} An important construction in Hodge theory is the \emph{monodromy filtration} of a nilpotent operator $N$ on an object $A$ in an abelian category. Given such an operator, there exists a unique increasing filtration $W(N)_\bullet A$ such that $N W(N)_\bullet A \subseteq W(N)_{\bullet-2}(A)$ and with the property that
\[ N^i \colon {\rm Gr}_i^{W(N)}(A) \to {\rm Gr}_{-i}^{W(N)}(A)\]
is an isomorphism for all $i > 0$.

If instead $A$ is itself already filtered by sub-objects $L_\bullet A \subseteq A$ in the abelian category $\cA$, there is the notion of \emph{relative monodromy filtration} for any nilpotent operator $N$ such that $NL_\bullet \subseteq L_{\bullet}$.

The relative monodromy filtration $W(N,L)_\bullet A$ is the unique, increasing filtration with the property that $N W(N,L)_\bullet \subseteq W(N,L)_{\bullet-2}$ and so that, for all $k\in \Z$ and $i\in \Z_{>0}$, the map
\[ N^i\colon {\rm Gr}_{k+ i}^{W(N,L)}{\rm Gr}^L_k(A) \to {\rm Gr}_{k-i}^{W(N,L)} {\rm Gr}^L_k(A)\]
is an isomorphism.

Such a filtration need not always exist, though an inductive criterion for existence is given in \cite{SteenbrinkZucker} (see also \cite{SaitoMHM}*{Lem. 1.2} and \cite{Kashiwara}*{Lem. 3.1.1}).

One of the most useful observations in the theory of relative monodromy filtrations is due to Kashiwara \cite{Kashiwara}*{Thm. 3.2.9}, which allows one to obtain a \emph{canonical splitting} of the filtration induced by $L_\bullet$ on ${\rm Gr}^{W(N,L)}_k(A)$ for any $k\in \Z$. This splitting (with an alternative proof) is also given in \cite{SaitoMHM}*{Prop. 1.5}. In other words, there is a canonical isomorphism
\begin{equation} \label{eq-canonicalSplitting} L_i {\rm Gr}^{W(N,L)}_k(A) \cong \bigoplus_{j\leq i} {\rm Gr}^L_j {\rm Gr}^{W(N,L)}_k(A).\end{equation}

The above definitions and results hold in any exact category. This extension is important when considering (bi)-filtered $\cD$-modules $(\cM,F,W)$, though it makes the notation a bit cumbersome. For the details on this extension, consult \cite{SaitoMHM}*{Ch. 1}.

In the theory of mixed Hodge modules, another property required of any $(\cM,F,W)$ which underlies a mixed Hodge module is the following: for any $\lambda \in [0,1]$, the relative monodromy filtration on ${\rm Gr}_V^{\lambda}(\cM)$ with respect to the nilpotent operator $t\de_t - \lambda +1$ for the filtration
\[ L_\bullet {\rm Gr}_V^\lambda(\cM) = \begin{cases} {\rm Gr}_V^0(W_\bullet \cM) & \lambda = 0 \\ {\rm Gr}_V^{\lambda}(W_{\bullet+1}\cM) & \lambda \in (0,1] \end{cases}\]
should exist. The shift by 1 in the case $\lambda \in (0,1]$ is forced in the theory, see \cite{SaitoMixedSheaves}*{Prop. 6.11}. This relative monodromy filtration is then the weight filtration on ${\rm Gr}_V^\lambda(\cM)$ as a mixed Hodge module.

\subsection{Higher codimension V-filtrations} \label{subsect-HigherCodimV}
In this subsection, we review the results of \cites{CD,CDS} concerning Koszul-like complexes of higher codimension $V$-filtrations. For ease, we work always on $T = X\times \A^r_t$ with coordinates $t_1,\dots, t_r$ on $\A^r_t$. By the graph embedding trick, this is always the local situation. 

If $\cM$ is a $\cD_{X \times \A^r_t}$-module underlying a mixed Hodge module, then it admits a $V$-filtration along $t_1,\dots, t_r$. For details on this $V$-filtration, see \cites{BMS, CD}. 

The $V$-filtration is the unique exhaustive decreasing, discretely and left continuously $\Q$-indexed filtration $V^\bullet \cM$ such that
\begin{enumerate} \item $V^\chi \cM$ is coherent over $V^0 \cD_{X\times \A^r_t} = \cD_X[t_1,\dots, t_r]\langle t_i\de_{t_j} \mid i,j \in \{1,\dots, r\}\rangle$.

\item $(t_1,\dots, t_r) V^\chi \cM \subseteq V^{\chi+1}\cM$ for all $\chi \in \Q$, with equality for $\chi \gg 0$.

\item $\de_{t_i} V^\chi \cM \subseteq V^{\chi-1}\cM$ for all $i \in \{1,\dots, r\}$ and $\chi\in \Q$.

\item Let $\theta = \sum_{i=1}^r t_i \de_{t_i}$. Then the operator $\theta - \chi +r$ is nilpotent on 
\[{\rm Gr}_V^\chi(\cM) = V^\chi \cM/V^{>\chi}\cM,\] where $V^{>\chi} \cM = \bigcup_{\beta > \chi} V^\beta \cM$.
\end{enumerate}

\begin{rmk} As in the hypersurface case, we will also use the operator $s = \sum_{i=1}^r s_i$ where $s_i = -\de_{t_i} t_i$. As $s = -(\theta+r)$, we can restate the last condition as requiring that $s+\chi$ is nilpotent on ${\rm Gr}_V^\chi(\cM)$.
\end{rmk}

As in the introduction, we define
\[A^\chi(\cM) = \cx{V^\chi\cM \xrightarrow[]{t_i} \bigoplus_{i=1}^r V^{\chi+1}\cM \xrightarrow[]{t_i} \dots \xrightarrow[]{t_i} V^{\chi+r}\cM},\]
\[B^\chi(\cM) = \cx{{\rm Gr}_V^\chi(\cM) \xrightarrow[]{t_i} \bigoplus_{i=1}^r {\rm Gr}_V^{\chi+1}(\cM) \xrightarrow[]{t_i} \dots \xrightarrow[]{t_i} {\rm Gr}_V^{\chi+r}(\cM)}\]
in degrees $0,\dots, r$ and
\[C^\lambda(\cM) = \left[ {\rm Gr}_V^{r+\lambda}(\cM) \xrightarrow[]{\de_{t_i}} \bigoplus_{i=1}^r {\rm Gr}_V^{r+\lambda-1}(\cM) \xrightarrow[]{\de_{t_i}} \dots \xrightarrow[]{\de_{t_i}} {\rm Gr}_V^{\lambda}(\cM) \right]\]
in degrees $-r,\dots, 0$.

We let ${\rm Kosz}(\cM,t) = [ \cM \xrightarrow[]{t_i} \bigoplus_{i=1}^r \cM \xrightarrow[]{t_i} \dots \xrightarrow[]{t_i} \cM]$ denote the Koszul complex on $\cM$ for the elements $t_1,\dots ,t_r$. 

\begin{prop} \label{prop-KoszulVFilt} Let $\sigma \colon X\times \{0\} \to X\times \A^r_t$ be the inclusion of the zero section. We have quasi-isomorphisms $\sigma^!(\cM) = {\rm Kosz}(\cM,t) \cong A^0(\cM) \cong B^0(\cM)$.

In fact, for $\chi \neq 0$, the complexes $B^\chi(\cM)$ and $C^\chi(\cM)$ are acyclic and for $\chi > 0$, the complex $A^\chi(\cM)$ is acyclic.
\end{prop}

The main results of \cites{CD,CDS} are extensions of the results of the \propositionref{prop-KoszulVFilt} to include the Hodge and weight filtrations.

For $(\cM,F,W)$ a bi-filtered $\cD_{X\times \A^r_t}$-module underlying a mixed Hodge module, we define filtered complexes
\[F_p A^\chi(\cM) = \cx{F_{p+r} V^\chi\cM \xrightarrow[]{t_i} \bigoplus_{i=1}^r F_{p+r} V^{\chi+1}\cM \xrightarrow[]{t_i} \dots \xrightarrow[]{t_i} F_{p+r}V^{\chi+r}\cM},\]
\[F_p B^\chi(\cM) = \cx{F_{p+r}{\rm Gr}_V^\chi(\cM) \xrightarrow[]{t_i} \bigoplus_{i=1}^r F_{p+r}{\rm Gr}_V^{\chi+1}(\cM) \xrightarrow[]{t_i} \dots \xrightarrow[]{t_i} F_{p+r}{\rm Gr}_V^{\chi+r}(\cM)},\]
\[F_p C^\chi(\cM) = \cx{F_p{\rm Gr}_V^{r+\chi}(\cM) \xrightarrow[]{\de_{t_i}} \bigoplus_{i=1}^r F_{p+1}{\rm Gr}_V^{r+\chi-1}(\cM) \xrightarrow[]{\de_{t_i}} \dots \xrightarrow[]{\de_{t_i}} F_{p+r}{\rm Gr}_V^{\chi}(\cM)}.\]

Using Verdier specialization (as in Subsection \ref{subsect-Specializations} below), we can easily see that the relative monodromy filtration for $N = \theta - \chi +r$ on ${\rm Gr}_V^\chi(\cM)$ for the filtration $L_\bullet {\rm Gr}_V^\chi(\cM) = {\rm Gr}_V^{\chi}(W_\bullet \cM)$ exists. We denote it by $W_\bullet {\rm Gr}_V^\chi(\cM)$. 

We have the sub-complexes $L_\bullet B^0(\cM) = B^0(W_\bullet \cM) \subseteq B^0(\cM)$ and $L_\bullet C^0(\cM) = C^0(W_\bullet \cM) \subseteq C^0(\cM)$, and since the morphisms in $B^0(\cM)$ and $C^0(\cM)$ also commute with $N$ (in the obvious way), we see that these morphisms preserve the relative monodromy filtration. Thus, we get a weight filtration
\[W_\bullet B^0(\cM) = \cx{W_\bullet {\rm Gr}_V^0(\cM) \xrightarrow[]{t_i} \bigoplus_{i=1}^r W_\bullet {\rm Gr}_V^{1}(\cM) \xrightarrow[]{t_i} \dots \xrightarrow[]{t_i} W_\bullet {\rm Gr}_V^{r}(\cM)},\]
\[W_\bullet C^0(\cM) = \cx{W_\bullet {\rm Gr}_V^r(\cM) \xrightarrow[]{\de_{t_i}} \bigoplus_{i=1}^r W_\bullet {\rm Gr}_V^{r-1}(\cM) \xrightarrow[]{\de_{t_i}} \dots \xrightarrow[]{\de_{t_i}} W_\bullet {\rm Gr}_V^{0}(\cM)}.\]

With these filtrations on the complexes, we have the following results:
\begin{thm}[\cite{CD}*{Thm. 1.1, Thm. 1.2}, \cite{CDS}*{Thm. 1}]\label{thm-CDResults} The complexes
\[ A^\chi(\cM,F) \text{ and } B^\chi(\cM,F) \text{ are filtered acyclic for all } \chi > 0\]
and\[ C^\chi(\cM,F) \text{ is filtered acyclic for all } \chi < 0.\]

Moreover, we have filtered quasi-isomorphisms
\[ \sigma^!(\cM,F) \cong A^0(\cM,F) \cong B^0(\cM,F),\]
and the latter two complexes are strictly filtered. Similarly,
\[ \sigma^*(\cM,F) \cong C^0(\cM,F)\]
and the right-most complex is strictly filtered.

Finally, the filtrations $W_\bullet \cH^i B^0(\cM,F)$ and $W_\bullet \cH^i C^0(\cM,F)$ induced by 
$W_\bullet B^0(\cM,F)$ and $W_\bullet C^0(\cM,F)$ respectively satisfy
\[ {\rm Gr}^W_k \cH^i B^0(\cM,F) \cong {\rm Gr}^W_{k+i} \cH^i \sigma^!(\cM,F),\]
\[{\rm Gr}^W_k \cH^{-i} C^0(\cM,F) \cong {\rm Gr}^W_{k-i} \cH^{-i} \sigma^*(\cM,F),\]
as filtered $\cD_X$-modules underlying polarizable Hodge modules.
\end{thm}

We remark that the filtered complexes $(B^0(\cM),W), (C^0(\cM),W)$ need not be strict. However, the weight spectral sequence degenerates at $E_2$, see, for example \cite{SaitoMHC}*{Prop. 2.3}.

\section{Generalized V-filtrations} \label{sect-LV} Throughout this section, we work on $T= X\times \A^r_t$, and let $t_1,\dots, t_r$ be the coordinates on $\A^r_t$. 

We call a linear form $L(s) = \sum_{i=1}^r a_i s_i$ a \textit{slope} if it is non-zero and $a_i \in \Z_{\geq 0}$ for all $i$. It is non-degenerate if $a_i\neq 0$ for all $i$. Given a slope $L$, we obtain a $\Z$-indexed filtration on $\cD_T$ by
\[ {}^L V^j \cD_T = \left\{\sum_{\beta,\gamma} P_{\beta,\gamma} t^\beta \de_t^\gamma \mid P_{\beta,\gamma} \in \cD_X, L(\beta) \geq L(\gamma) + j.\right\}\]

If $\cM$ is a module over $\cD_T$, we say that a $\Z$-indexed filtration $U^\bullet \cM$ is \emph{compatible} if \[{}^L V^j \cD_T \cdot U^k \cM \subseteq U^{k+j}\cM,\]
for example, the filtration ${}^L V^\bullet \cD_T$ is compatible (in other words, it is a multiplicative filtration).

We define the \emph{Rees ring} $R_L(\cD_T) = \bigoplus_{k \in \Z} {}^L V^k \cD_T u^{-k}$. It is a $\Z$-graded ring, by the multiplicative property of the filtration. Given a module $\cM$ with a compatible filtration $U^\bullet \cM$, we define the \emph{Rees module} $R_U(\cM) = \bigoplus_{k\in \Z} U^k \cM u^{-k}$. The filtration $U^\bullet \cM$ is \emph{good} if $R_U(\cM)$ is coherent over $R_L(\cD_T)$. We will also say that $(\cM,U)$ is a good filtered $(\cD_T, {}^L V)$-module in this case.

For a filtered module $(\cM,U)$ and $k\in \Z$, we set $(\cM,U[k])$ for the filtration with $U[k]^\bullet \cM = U^{\bullet-k} \cM$. The following is immediate.
\begin{lem} An exhaustive filtration $U^\bullet \cM$ is good if and only if there exist $m_1,\dots, m_N$ and $k^{(1)},\dots, k^{(N)} \in \Z$ such that we have
\[U^\bullet \cM = \sum_{i=1}^N {}^L V^{\bullet - k^{(i)}} \cD_T \cdot m_i.\]

Equivalently, we have a strict surjection $\bigoplus_{i=1}^N (\cD_T,{}^L V[k^{(i)}]) \to (\cM,U)$.
\end{lem}

From this, we get comparability of good filtrations:
\begin{lem} \label{lem-goodFiltComparable} Let $U_1^\bullet \cM, U_2^\bullet \cM$ be two good filtrations of the $\cD_T$-module $\cM$. Then there exists $k \in \Z_{\geq 0}$ such that
\[ U^{\bullet + k}_1 \cM \subseteq U_2^\bullet \cM \subseteq U_1^{\bullet-k}\cM.\]
\end{lem}

The following can be thought of as the analogue of the Artin-Rees lemma:
\begin{lem} \label{lem-ArtinRees} The ring $R_L(\cD_T)$ is Noetherian. Thus, if $(\cM,U)$ is a good filtered $(\cD_T,{}^L V)$-module and $\cN \subseteq \cM$ is a sub $\cD_T$-module, the induced filtration
\[ U^\bullet \cN = \cN \cap U^\bullet \cM \text{ is good}.\]
\end{lem}
\begin{proof} The second claim is immediate from the first. The first follows from \lemmaref{lem-LArtinRees} below, which we postpone until we discuss specialization constructions.
\end{proof}

\begin{defi} Let $(\cM,U)$ be a good filtered $(\cD_T,{}^L V)$-module. The $b$-function for $U^\bullet \cM$ is the monic polynomial $p(w) \in \C[w]$ of least degree (if some such non-zero polynomial exists) such that
\begin{equation} \label{cond-Lspecializable} p(L(s) + k) U^k \cM \subseteq U^{k+1}\cM,\end{equation}
where $L(s) = \sum_{i=1}^r -a_i \de_{t_i} t_i$. 

We say $(\cM,U)$ is \emph{specializable} if it admits a $b$-function. For any subfield $A\subseteq \C$, we say $(\cM,U)$ is $A$-specializable if the $b$-function splits into linear factors over $A$.
\end{defi}

\lemmaref{lem-goodFiltComparable} can be used to show the following, which says that being $A$-specializable is a property of the module, not the filtration:
\begin{prop} If $(\cM,U)$ is a good filtered $(\cD_T,{}^L V)$-module which is $A$-specializable, then any other good filtration is also $A$-specializable.
\end{prop}
\begin{proof} Let $U_1^\bullet \cM, U_2^\bullet \cM$ be two good filtrations and assume $U_1^\bullet$ is $A$-specializable with $b$-function $b(w) \in A[w]$ (splitting completely over $A$). By \lemmaref{lem-goodFiltComparable}, there exists $k$ such that $U_1^{\bullet+k} \cM \subseteq U_2^\bullet \cM \subseteq U_1^{\bullet-k} \cM$.

Then
\[ b(L(s) + (j+k)) \dots b(L(s) + (j-k)) U_2^j \subseteq b(L(s) + (j+k)) \dots b(L(s) + (j-k)) U_1^{j-k}\]
but the right hand side, by definition of the $b$-function, is contained in $U_1^{j+1+k} \subseteq U_2^{j+1}$ hence $U_2^{\bullet}$ admits a $b$-function which divides the product $b(w + k) \dots b(w-k)$. Thus, $U_2^\bullet$ is $A$-specializable as that product splits completely over $A$.
\end{proof}

\lemmaref{lem-ArtinRees} shows that if $\cN\subseteq \cM$ is a submodule and $\cM$ is $A$-specializable, then $\cN$ is, too. This applies in particular to $\cD_T\cdot u \subseteq \cM$ for any element $u\in \cM$, which leads to the following definition.

\begin{defi} Let $\cM$ be $A$-specializable. For any $u \in \cM$, the $b$-function of $u$ is the monic polynomial of least degree $b(w) \in \C[w]$ such that
\[ b(L(s)) u \in {}^L V^1\cD_T \cdot u,\]
which we denote by $b_u(w)$. Such a polynomial exists (and splits over $A$) for any section $u\in \cM$ of an $A$-specializable module.
\end{defi}

For the remainder, we assume $\cM$ is $\Q$-specializable, though the same constructions can be made with $A = \R$. Given $u\in \cM$, let $b_u(w)$ be the $b$-function of $u$, which we factor as
\[ b_u(w) = (w+\gamma_1)\dots (w+\gamma_N),\]
with $\gamma_1 \leq \dots \leq \gamma_N$. Then define the \emph{$L$-order} of $u$ to be ${\rm ord}_L(u) = \gamma_1$.

This leads to a $\Q$-indexed filtration ${}^L V^\bullet \cM$ defined by
\[ {}^L V^\lambda \cM = \{u\in \cM \mid {\rm ord}_L(u)\geq \lambda\},\]
whose $\Z$-indexed part is characterized by the following proposition.

\begin{prop} \label{prop-orderFilt} Let $\cM$ be $\Q$-specializable. Then there exists a unique good filtration $U^\bullet \cM$ whose $b$-function satisfies the following: $b_U(-\gamma) = 0 \implies \gamma \in [0,1)$.

Moreover, for all integers $j\in \Z$, this filtration satisfies
\[ {}^L V^j \cM = U^j \cM.\]
\end{prop}
\begin{proof} This is standard, for example, the same argument is given in \cite{MHMProj}*{Lem. 9.3.16}. See also \cite{Sabbah1}*{Pg. 307}.
\end{proof}

\begin{rmk} \label{rmk-refineQIndexed} The $\Z$-indexed filtration in the proposition statement can be refined in the following way to a $\Q$-indexed filtration, which agrees with the $\Q$-indexed filtration ${}^L V^\bullet \cM$. The main idea is to lift generalized eigenspaces of the operator $L(s)$ on the associated graded pieces ${\rm Gr}_U^j(\cM)$.

To be precise, let $U^\bullet \cM$ be a good $\Z$-indexed filtration whose $b$-function $b_U(w)$ satisfies $b_U(-\gamma) = 0$ implies $\gamma \in [0,1)$. Write $b_U(w) = (w+\gamma_1)^{m_1}\dots (w+\gamma_N)^{m_N}$ for $0 \leq \gamma_1\leq \gamma_2\leq \dots \leq \gamma_N < 1$. 

For any $j\in \Z$, define $U^{\gamma_N +j}\cM = \{u \in \cM \mid (L(s)+\gamma_N +j)^{m_N}u \in U^{j+1}\cM\}$. Inductively, we then define
\[ U^{\gamma_i +j}\cM = \{u\in \cM \mid (L(s)+\gamma_i +j)^{m_i} u \in U^{j+\gamma_{i+1}}\cM\}.\]

For any $\chi \in \Q$, set $j = \lfloor \chi \rfloor$ and $\varepsilon = \chi - j \in [0,1)$. First, if $\varepsilon \leq \gamma_N$, let $i$ be minimal such that $\varepsilon \leq \gamma_i$ and set $U^\chi \cM = U^{j+\gamma_i}\cM$. Otherwise, if $\varepsilon > \gamma_N$, set $U^\chi \cM = U^{j+1+\gamma_1}\cM$.

It is an easy exercise to see that, in this case, $U^\bullet \cM$ is a decreasing, discrete and left-continuous $\Q$-indexed filtration.
\end{rmk}

We call ${}^L V^\bullet \cM$ ``the (canonical) ${}^L V$-filtration" of $\cM$. For example, for any $I \subseteq \{1,\dots, r\}$, if $L = \sum_{i\in I} s_i$, this filtration is the Kashiwara-Malgrange $V$-filtration of $\cM$ along the subvariety $V(t_i\mid i\in I)$.

Define Koszul-like complexes
\[ A_L^\gamma(\cM) = \cx{{}^L V^\gamma \cM \xrightarrow[]{t_i} \bigoplus_{i=1}^r {}^L V^{\gamma+a_i}\cM \xrightarrow[]{t_i} \dots \xrightarrow[]{t_i} {}^L V^{\gamma+|L|}\cM}\]
\[ B_L^\gamma(\cM) = \cx{{\rm Gr}_L^\gamma (\cM) \xrightarrow[]{t_i} \bigoplus_{i=1}^r {\rm Gr}_L^{\gamma+a_i}(\cM) \xrightarrow[]{t_i} \dots \xrightarrow[]{t_i} {\rm Gr}_L^{\gamma+|L|}(\cM)}\]
placed in degree $0,1,\dots,r$. Define also
\[ C_L^\gamma(\cM) = \cx{{\rm Gr}_L^{\gamma+|L|} (\cM) \xrightarrow[]{\de_{t_i}} \bigoplus_{i=1}^r {\rm Gr}_L^{\gamma+|L|-a_i}(\cM) \xrightarrow[]{\de_{t_i}}  \dots \xrightarrow[]{\de_{t_i}}  {\rm Gr}_L^{\gamma}(\cM)}\]
placed in degree $-r,-r+1,\dots,0$.

The following can be shown in the same way as \propositionref{prop-KoszulVFilt} above, we give a sketch of the argument for the reader's convenience.

\begin{lem} \label{lem-DModuleRestriction} Let $\cM$ be $\Q$-specializable. Then the complexes $B^\chi_L(\cM)$ and $C^\chi_L(\cM)$ are acyclic for all $\chi \neq 0$. Moreover, $A^\chi_L(\cM)$ is acyclic for all $\chi > 0$.

Let $\cM$ admit an ${}^L V$-filtration. Then
\[ \sigma^!(\cM) \cong A^0_L(\cM) \cong B^0_L(\cM)\]
for any slope $L$.
\end{lem}

\begin{proof}[Sketch of Proof] As $L(\theta) - \chi -j +|L|$ is nilpotent on ${\rm Gr}_L^\chi(\cM)$, we see that $L(\theta) -j + |L|$ is an automorphism of ${\rm Gr}_L^{\chi+j}(\cM)$ for all $\chi \neq 0$. This allows us to define an automorphism of the complex $B^\chi_L(\cM)$ (resp. $C^\chi_L(\cM)$). But using the $\de_{t_i}$ maps (resp. $t_i$), it is easy to see that this automorphism is null-homotopic, proving that $B^\chi_L(\cM)$ (resp. $C^\chi_L(\cM)$) is acyclic for $\chi \neq 0$. See \cite{CD}*{Thm. 3.1} for details on constructing the null-homotopy.

The claim that $A^\chi(\cM)$ is acyclic is easy to show for $\chi \gg 0$ using a strict surjection (on the $\Z$-indexed part)
\[ \bigoplus_I (\cD_{X\times \A^r},V[\beta_i]) \to (\cM,V),\]
following the argument of \lemmaref{lem-EventuallyFiltAcyclicGood} below. As the argument is exactly analogous and the later one is more detailed, we omit the proof here.

Finally, the definition of $\sigma^!(\cM)$ is as the derived $\cO$-module pull-back of $\cM$ along $\sigma \colon X \times \{0\} \to X\times \A^r_t$. Using the Koszul resolution of $\cO_{X\times \{0\}}$, we see that $\sigma^!(\cM) = {\rm Kosz}(\cM,t)$.

Using what we have already shown, it is obvious that $A^0(\cM) \to B^0(\cM)$ is a quasi-isomorphism. By discreteness of the filtration $V^\bullet \cM$, we can check that for all $\chi < 0$, the inclusion $A^0(\cM) \to A^\chi(\cM)$ is a quasi-isomorphism. Taking the inductive limit as $\chi \to -\infty$ (which is an exact functor) proves the claim, as $\bigcup V^\lambda \cM = \cM$.
\end{proof}

The following proposition gives the characterizing properties of the canonical ${}^L V$-filtration, as well as a useful test for containment. As above, instead of $L(s)+\chi$ being nilpotent on ${\rm Gr}_L^{\chi}(\cM)$, we could ask for $\theta_L - \chi + |L|$ to be nilpotent, where $\theta_L = \sum_{i=1}^r a_i t_i\partial_{t_i}$.

\begin{prop} \label{prop-characterizeLV} Let $\cM$ be a $\cD_T$-module. Assume $U^\bullet \cM$ is an exhaustive discrete, left-continuous $\Q$-indexed filtration such that the following conditions hold:
\begin{enumerate} \item ${}^L V^k \cD_T \cdot U^\chi \cM \subseteq U^{k+\chi}\cM$,
\item for $\chi \gg 0$ we have equality $U^{\chi}\cM = \sum_{i=1}^r t_i U^{\chi-a_i}\cM,$
\item for any $\chi \in \Q$, the operator $L(s)+\chi$ is nilpotent on ${\rm Gr}_U^\chi(\cM)$.
\end{enumerate}

Then ${}^L V^{\chi}\cM \subseteq U^\chi \cM$ for all $\chi \in \Q$. If, moreover, we assume that $U^\chi \cM$ is coherent over ${}^L V^0\cD_T$ for all $\chi \in \Q$, then equality holds.
\end{prop}
\begin{proof} The claim follows from the observation that if $U_1^\bullet \cM$ is a filtration satisfying all conditions in the proposition statement and the coherence over ${}^L V^0\cD_T$ assumption and $U_2^\bullet \cM$ is another filtration satisfying just the conditions in the proposition statement, then $U_1^\bullet \cM \subseteq U_2^\bullet \cM$. This can be shown similarly to \cite{CDMO}*{Thm. 3.3} and we leave the checking of details to the reader.

To see why this observation implies the desired result, it suffices to note that ${}^L V^\bullet \cM$ satisfies all the conditions in the proposition statement. Indeed, by construction, $L(s)+\chi$ is nilpotent on ${\rm Gr}_L^\chi(\cM)$. By definition of ${}^L V^\chi \cM$ in Remark \ref{rmk-refineQIndexed}, it is easy to see that the $\Q$-indexed filtration ${}^L V^\chi \cM$ is compatible, using the fact that $P t^\beta \de_t^\gamma L(s) = L(s+\beta-\gamma) P t^\beta \de_t^\gamma$ for $P\in \cD_X$. By the goodness of the $\Z$-indexed filtration ${}^L V^\bullet \cM$ and Noetherianity of the ring ${}^L V^0\cD_T$ (which holds because it is the $0$th graded piece of the Noetherian $\Z$-graded ring $R_L(\cD_T)$), we see that each ${}^L V^\chi \cM$ is ${}^L V^0\cD_T$-coherent.

From this, using the acyclicity of \lemmaref{lem-DModuleRestriction}, we see that the Koszul-like complex
\[ A_L^\gamma(\cM) = \cx{{}^L V^\gamma \cM \xrightarrow[]{t_i} \bigoplus_{i=1}^r {}^L V^{\gamma+a_i}\cM \xrightarrow[]{t_i} \dots \xrightarrow[]{t_i} {}^L V^{\gamma+|L|}\cM}\]
is acyclic for all $\gamma > 0$, where $|L| = \sum_{i=1}^r a_i$. In particular, by the vanishing of the rightmost cohomology, we see that for all $\chi >|L|$, we have the equality
\[ {}^L V^\chi \cM = \sum_{i=1}^r t_i {}^L V^{\chi-a_i}\cM.\]
\end{proof}

\begin{eg} Let $\cM$ be supported on $V(t_1,\dots, t_r) \subseteq T$. Then we can write 
\[ \cM = \bigoplus_{\alpha \in \N^r} \cM_0 \de_t^\alpha \delta_0 \] 
where $t_i(\eta \delta_0) = 0$ for all $\eta \in \cM_0$. In particular, $$L(s) \eta \delta_0 = \sum_{i=1}^r -a_i \de_{t_i} t_i (\eta \delta_0) = 0.$$
It is not hard to check that $\cM$ is $\Q$-specializable. We have the equality 
\[L(s) \de_t^\alpha\eta\delta_0 = \de_t^\alpha L(s+\alpha)\eta\delta_0 = L(\alpha) \de_t^\alpha\eta\delta_0,\] where the last equality is due to the vanishing $L(s)\eta\delta_0 = 0$.

Hence, we see in fact that
\[ {}^L V^{\lambda}\cM = {}^L V^{\lceil \lambda \rceil}\cM = \bigoplus_{L(\alpha)\leq -\lceil \lambda \rceil} \cM_0 \de_t^\alpha \delta_0,\]
which has $b$-function equal to $b(w) = w$.
\end{eg}

\begin{eg} \label{eg-SupportedHypersurface} Let $\cM$ be a coherent $\cD_T$-module supported on $V(t_i) \subseteq T$. Then we have $\cM = \bigoplus_{j \in \N} \cN \de_{t_i}^j \delta_0$ where $t_i(\eta \delta_0) = 0$ for all $\eta \in \cN$. In particular,
we have $L(s)(\eta \delta_0) = (\ell(s) \eta)\delta_0$, where if $L = \sum_{i=1}^r a_i s_i$ then $\ell = \sum_{j\neq i} a_j s_j$. More generally,
\[ L(s)(\eta \de_{t_i}^j \delta_0) = (\ell(s) + a_i j)(\eta) \de_{t_i}^j \delta_0.\]

From this it is easy to see that $\cM$ is $\Q$-specializable if $\cN$ is, and that we have equality
\[ {}^L V^\lambda \cM = \bigoplus_{j\geq 0} {}^\ell V^{\lambda + a_i j}\cN \de_{t_i}^j \delta_0.\]
\end{eg}

The following lemma applies in particular to the case when $\cM$ is $\cO$-coherent, or when $\cM = \cN \boxtimes \O_{\A^r_t}$ for $\cN$ a $\cD_X$-module.

\begin{lem} \label{lem-LVtAdic} Let $\cM$ be a coherent $\cD_T$-module. Then the ${}^L V$-filtration is ``$t$-adic", in the sense that
\[ {}^L V^{\lambda}\cM = \begin{cases} \cM & \lambda \leq |L| \\ \sum_{i=1}^r t_i {}^L V^{\lambda-a_i}\cM & \lambda > |L|,\end{cases}\]
if and only if $\cM$ is coherent over ${}^L V^0\cD_T$.
\end{lem}
\begin{proof} If the filtration is $t$-adic, then ${}^L V^{|L|}\cM = \cM$ is, by definition, coherent over ${}^L V^{0}\cD_T$. Note that another way to write the $t$-adic filtration is
\[ {}^L V^{\lambda}\cM = \big(t^\beta \mid L(\beta + \underline{1}) \geq \lambda\big)\cdot \cM.\]

For the converse, note that if $\cM$ is coherent over ${}^L V^0\cD_T$ and we define
\[ U^\lambda \cM = \big(t^\beta \mid L(\beta+\underline{1})\geq \lambda\big)\cdot \cM,\]
then each $U^\lambda \cM$ is coherent over ${}^L V^{0}\cD_T$. Then we need to check the remaining properties of the ${}^L V$-filtration to conclude. 

Assume $L(\beta+\underline{1})\geq \lambda$. Then $L(\beta+e_i +\underline{1})\geq \lambda +a_i$, so we see that $t_i U^\lambda \cM \subseteq U^{\lambda+a_i}\cM$. 

Let $L(\beta+\underline{1}) \geq \lambda$. Then for $t^\beta m \in U^\lambda \cM$, when we apply $\de_{t_i}$, there are two options. Either $\beta_i = 0$, in which case
\[ \de_{t_i}(t^\beta m) = t^\beta(\de_{t_i}m) \in U^\lambda\cM \subseteq U^{\lambda-a_i}\cM,\]
or $\beta_i > 0$, in which case
\[ \de_{t_i}(t^\beta m) = t^\beta(\de_{t_i} m) + \beta_i t^{\beta-e_i}m,\]
and so since $L(\beta - e_i + \underline{1}) \geq \lambda -a_i$, this shows that $\de_{t_i} U^{\lambda}\cM \subseteq U^{\lambda-a_i}\cM$.

Each $U^\lambda \cM$ is clearly stable by $\cD_X$, so we only need to prove the nilpotency of $L(s)+\lambda$ on ${\rm Gr}^\lambda_U \cM$. Note that ${\rm Gr}^\lambda_U \cM \neq 0$ if and only if there exists $\beta \in \N^r$ with $L(\beta+\underline{1}) = \lambda$. Take such a $\beta$ and consider an element $t^\beta m$. Then
\[ (L(s)+\lambda)(t^\beta m) = L(s+\beta+\underline{1})(t^\beta m) = t^\beta L(s+\underline{1})m,\]
and $s+\underline{1} = (s_1+1,\dots, s_r+1) = (-t_1 \de_{t_1},\dots, -t_r \de_{t_r})$. Hence,
\[ L(s+\underline{1}) m = -\sum_{i=1}^r a_i t_i \de_{t_i}(m),\]
and we have that $a_i t^{\beta+e_i}(\de_{t_i}m) \in U^{>\lambda}\cM$, proving the claim.
\end{proof}

\subsection{Specialization Constructions} \label{subsect-Specializations} In this section, we use deformation to the normal bundle to allow us to use established tools to study ${}^L V$-filtrations. The idea goes back at least to Verdier. The papers \cite{BMS} and \cite{CD} use these ideas to study the $V$-filtration when $L = \sum_{i=1}^r s_i$, and Wu's article \cite{Wu} discusses the case of arbitrary slopes $L$.

As above, consider $T = X \times \A^r_t$. Let $L = \sum_{i=1}^r a_i s_i$ be a non-degenerate slope. For $T = X\times \A^r_t$, define
\[ \widetilde{T}^L = X\times \A^r_{z} \times \A^1_u,\]
where the coordinates on the $\A^r$ term are $z_1,\dots, z_r$. We should think of $z_ i = t_i/u^{a_i}$, in the sense that we consider the natural morphism
\[ \widetilde{T}^L \to X\times \A^r_t, \quad (x,z,u) \mapsto (x, z_1u^{a_1},\dots, z_r u^{a_r}).\]

Over $\{u\neq 0\}$, this morphism is isomorphic to the projection $T \times \mathbf G_m \to T$. Let 
\[j_L \colon \{u\neq 0\} \to \widetilde{T}^L \] 
be the open embedding.
Then $\widetilde{T}^L$ is naturally a deformation to the normal bundle of $X\times \{0\} \subseteq T$.
The projection $\widetilde{T}^L \to \A^1_u$ is clearly smooth, so we can consider the ring of relative differential operators $\cD_{\widetilde{T}^L/\A^1_u}$.

We have the following identification of the relative differential operators with the Rees ring of $\cD_T$ for the ${}^L V$-filtration. As relative differential operator rings are Noetherian, this proves the first claim of \lemmaref{lem-ArtinRees}.

\begin{lem} \label{lem-LArtinRees} We have a filtered isomorphism of rings
\[ (\cD_{\widetilde{T}^L/\A^1_u},F) \cong (R_{L}(\cD_T),F).\]

Hence, we also have an isomorphism
\[ R_F(\cD_{\widetilde{T}^L/\A^1_u}) \cong R_{F,L}(\cD_T) = \bigoplus_{k,j} F_k {}^L V^j\cD_T z^k u^{-j}.\]

In particular, the rings $R_L(\cD_T)$ and $R_{F,L}(\cD_T)$ are Noetherian.
\end{lem}
\begin{proof} We can reduce to the case that $X$ is a point. Indeed, it is easy to see that if the result is established in that setting, then using that $X = X\times {\rm Spec}(\C)$, we have isomorphisms
\[ (\cD_{\widetilde{T}^L/\A^1_u},F) \cong (\cD_X,F) \boxtimes (\cD_{\widetilde{T}^L_{\rm pt}/\A^1_u},F) \cong (\cD_X,F) \boxtimes (R_L(\cD_{T_{\rm pt}},F)) \cong (R_L(\cD_T),F),\]
where $T_{\rm pt} = {\rm Spec}(\C) \times \A^r_t$ and $\widetilde{T}^L_{\rm pt} = {\rm Spec}(\C) \times \A^r_z \times \A^1_u$, we conclude the general result.

In this case, we have $\cD_{\widetilde{T}^L_{\rm pt}/\A^1_u} = \C[u,z_1,\dots, z_r]\langle \de_{z_1},\dots, \de_{z_r}\rangle$. Define a $\C[u]$-linear morphism
\[ \cD_{\widetilde{T}^L/\A^1_u} \to R_L(\cD_T), \quad z_i \mapsto \frac{t_i}{u^{a_i}},\quad \de_{z_i} \mapsto \de_{t_i} u^{a_i}.\]
Define the inverse by
\[ R_L(\cD_T) \to \cD_{\widetilde{T}^L/\A^1_u}, \quad t^\beta \de_t^\gamma u^{L(\gamma-\beta)} \mapsto z^\beta \de_z^\gamma,\]
and extend it $\C[u]$-linearly. These maps are easily checked to be mutually inverse and it is clear that they preserve the order filtration.

For the final claim, note that $R_F(-)$ applied to any ring of relative differential operators (for a smooth map between smooth varieties) is Noetherian. One way to see this is that, by taking the quotient by the central element $z$, one gets the associated graded ring, which is naturally identified with the structure sheaf of the relative cotangent bundle as in the case of usual differential operators.
\end{proof}

Let $\cM$ be a $\cD_T$-module with $\Z$-indexed filtrations $F_\bullet \cM, U^\bullet \cM$ such that $F_\bullet \cM$ is compatible with $F_\bullet \cD_T$ and $U^\bullet \cM$ is compatible with ${}^L V^\bullet \cD_T$. We can define
\[ F_p A^\chi(\cM,U) = \cx{F_{p+r} U^\chi \cM \xrightarrow[]{t_i} \bigoplus_{i=1}^r F_{p+r} U^{\chi+a_i} \cM \xrightarrow[]{t_i} \dots \xrightarrow[]{t_i} F_{p+r} U^{\chi + |L|}\cM}.\]

The Noetherianity of the ring $R_{F,L}(\cD_T)$ allows us to prove $F$-filtered acyclity of $A^\chi(\cM,U,F)$ for $\chi \gg 0$.

\begin{lem} \label{lem-EventuallyFiltAcyclicGood} Let $(\cM,U,F)$ be a bi-filtered $\cD_T$-module as above such that $R_{F,U}(\cM) = \bigoplus_{k,j\in \Z} F_k U^j \cM z^k u^{-j}$ is coherent over $R_{F,L}(\cD_T)$. Then there exists $k_0\in \Z$ such that $k \geq k_0$ implies $A^k(\cM,U,F)$ is $F$-filtered acyclic.
\end{lem}
\begin{proof} It is a simple computation to see that $A^k(\cD_T,{}^L V,F)$ is filtered acyclic for all $k \geq 0$. Indeed, by taking ${\rm Gr}^F_\bullet$, this boils down to the claim that variables in a polynomial ring form a regular sequence, hence the corresponding Koszul complex is acyclic (except at the right). For the vanishing of the right-most cohomology, use that $k \geq 0$.

We will prove that, for any $(\cM,U,F)$ as in the lemma statement, there exists $k_0(\cM,U,F,j)$ (depending also on $0 \leq j \leq r$) with the property that, for $k \geq k_0(\cM,U,F,j)$, we have
\[ \cH^j A^k(\cM,U,F) = 0. \]

This will be done by descending induction on $j$.

Now, by coherence of $R_{F,U}(\cM)$ over $R_{F,L}(\cD_T)$, we (locally) have a finite indexing set $I$ and a strict surjection
\[ \bigoplus_{i\in I} (\cD_T,{}^L V[b_i],F[c_i]) \to (\cM,U,F) \to 0,\]
where $b_i,c_i \in \Z$ are integer shifts of the $\Z$-indexed filtrations ${}^L V^\bullet \cD_T, F_\bullet \cD_T$. This strict surjection is equivalent to the existence of $\{m_i\}_{i\in I}$ satisfying
\[ F_p U^j \cM = \sum_{i\in I} F_{p-c_i} {}^L V^{j+b_i} \cD_T \cdot m_i\]
for all $p,j \in \Z$.

By Noetherianity of $R_{F,L}(\cD_T)$, the kernel $\cK$ with its induced filtrations $F_\bullet\cK$ and $U^\bullet\cK$ also satisfies $R_{F,U}(\cK)$ is coherent over $R_{F,L}(\cD_T)$. We have the $F$-strict short exact sequence of complexes
\[ 0 \to A^k(\cK,U,F) \to \bigoplus_{i\in I} A^{k-b_i}(\cD_T,{}^L V, F[c_i]) \to A^k(\cM,U,F) \to 0,\]
and recall that for $k \geq k_0 = \max\{b_i\}$, the middle term is filtered acyclic by the first sentence of the proof. Hence, by looking at the long exact sequence in cohomology, we get the vanishing of $\cH^r F_p A^k(\cM,U)$ for all $p$ and filtered isomorphisms
\[ \cH^jF_pA^k(\cM,U)\cong \cH^{j+1}F_p A^k(\cK,U)\]
for $0 \leq j < r$.

This proves the existence of $k_0(\cM,U,F,r)$ for any $(\cM,U,F)$. Moreover, we see by descending induction on $j$ that we can define $k_0(\cM,U,F,j) = \max\{ k_0(\cM,U,F,r), k_0(\cK,U,F,j+1)\}$, and so the claim has been shown by descending induction on $j$.
\end{proof}

Given a system of coordinates $x_1,\dots, x_n$ on $X$, the variety $\widetilde{T}^L$ has local coordinates $x_1,\dots, x_n, z_1 = \frac{t_1}{u^{a_1}},\dots, z_r = \frac{t_r}{u^{a_r}}, u$. The open subset $T \times \mathbf G_m$ has the simpler system of coordinates $x_1,\dots, x_n, t_1,\dots, t_r, u$, and the change of variables formula yields (using $\widetilde{(-)}$ to denote the functions $x_1,\dots, x_n, u$ viewed in the second system of coordinates):
\[ \de_{\widetilde{x_i}} = \de_{x_i}, \quad \de_{t_i} = \frac{1}{u^{a_i}} \de_{z_i},\]
\[ \de_{\widetilde{u}} = \de_u + \sum_{j=1}^r \de_{\widetilde{u}}(z_j) \de_{z_j} = \de_u - \sum_{j=1}^r a_j \frac{t_j}{u^{a_j+1}} \de_{z_j} = \de_{v_i} - \frac{1}{u} L(t\de_t).\]

For $M$ a mixed Hodge module on $T$, consider $\widetilde{M}_L = j_{L*}(M \boxtimes \Q^H_{\mathbf G_m}[1])$ a mixed Hodge module on $\widetilde{T}_L$. The underlying $\cD$-module $\tcM_L$ is the $\cO$-module
\[ \bigoplus_{k\in \Z} \cM u^k,\]
on which, thanks to the computation of the coordinate change above, the action is given by
\[ z_i(mu^k) = (t_i m)u^{k-a_i},\]
\[ \de_{z_i}(m u^k) = \de_{t_i}(m) u^{k+a_i},\]
\[ \de_{u}(m u^k) = (k + L(t\de_t))(m)u^{k-1}.\]

We have the following:
\begin{prop} \label{prop-VFiltLSpec} Let $V^\bullet \tcM_L$ be the $V$-filtration along $u$. Then
\[ V^\lambda \tcM_L = \bigoplus_{k\in \Z} {}^L V^{\lambda +|L| - k -1} \cM u^k.\]

For any $\lambda \geq 0$ and $p\in \Z$, we have
\[ F_p V^\lambda \tcM_L = \bigoplus_{k\in \Z} F_p {}^L V^{\lambda+|L|-k-1}\cM u^k.\]
\end{prop}
\begin{proof} The second claim follows from the first using the fact that for all $p \in \Z$ and $\lambda \geq 0$, we have by \lemmaref{lem-strictSupportHodge} equality
\[ F_p V^\lambda \tcM_L = V^\lambda \tcM_L \cap j_{L*}(F_p(\cM \boxtimes \cO_{\mathbf G_m})).\]

For the first claim, define 
\[ U^\lambda \tcM_L = \bigoplus_{k\in \Z} {}^L V^{\lambda + |L| -k -1} \cM u^k.\]

We show that $U^\lambda \tcM_L$ satisfies the properties of the $V$-filtration along $u$. A simple computation shows 
\[ u U^\lambda = U^{\lambda+1}, \quad \de_u U^\lambda \subseteq U^{\lambda-1}.\]

As
\begin{equation} \label{eq-EulerRelated} u\de_u(mu^k) = (k + L(t\de_t))(m)u^k\end{equation}
it is easy to see that $u\de_u - \lambda +1$ is nilpotent on ${\rm Gr}_U^\lambda(\tcM_L)$. By \propositionref{prop-characterizeLV}, this proves the containment $V^\lambda \tcM_L \subseteq U^\lambda \tcM_L$.

For the other containment, for any fixed $k$, define a filtration
\[ \cU^\lambda \cM = \{m\in \cM \mid m u^k \in V^{\lambda + k +1 - |L|}\tcM_L\}.\]
Note that it is enough to verify the conditions of \propositionref{prop-characterizeLV}, as this shows
\[ {}^L V^\bullet \cM \subseteq \cU^\bullet \cM,\]
which implies $\cU^{\lambda}\tcM_L \subseteq V^\lambda \tcM_L $ since these are the graded pieces.

Let $mu^k \in V^{\lambda + k + 1 - |L|}\tcM_L$. Then by applying $u^{a_i} z_i$, we see that \[(t_i m)u^k \in V^{\lambda + a_i +k+1-|L|}\tcM_L,\] so that $t_i \cU^\lambda \subseteq \cU^{\lambda+a_i}$. Applying $\de_{z_i}$, we get that $\de_{t_i}(m) u^{k+a_i} \in V^{\lambda+a_i + k + 1 - |L|}\tcM_L$. 

As $u$ acts invertibly on $\tcM_L$, we know that $u \colon V^\chi \tcM_L \to V^{\chi+1}\tcM_L$ is an isomorphism. Thus, having $\de_{t_i}(m) u^{k+a_i} \in V^{\lambda+k+1-|L|}\tcM_L$ implies that $\de_{t_i}(m) u^{k} \in V^{\lambda-a_i + k + 1 - |L|}\tcM_L$.

By definition of the $V$-filtration, we know that $V^{\lambda+ k+1-|L|}\tcM_L$ is coherent over $V^0\cD_{\widetilde{T}^L} = \cD_X[z,u]\langle \de_z,u\de_u\rangle$. For some fixed $\lambda$, choose generators $m_1u^{\ell_1},\dots, m_a u^{\ell_a}$ for $V^{\lambda+ k+1-|L|}\tcM_L$ over $V^0$. As $u \colon V^\chi \tcM_L \to V^{\chi+1}\tcM_L$ is an isomorphism for all $\chi \in \Q$, we see then that
\[m_1 u^{\ell_1 + j},\dots, m_a u^{\ell_a +j}\]
are generators of $V^{(\lambda+j)+k + 1 - |L|}\tcM_L$ for all $j\in \Z$.

Let $\ell = \min\{\ell_1,\dots, \ell_a\}$. Note that the only operators in $V^0$ which decrease the power of $u$ are $z_1,\dots, z_r$. Thus, we see that for any $b < \ell + j$, we have
\[ m u^b \in V^{(\lambda+j)+ k +1 - |L|}\tcM_L \implies m u^b \in (z_1,\dots, z_r) V^{(\lambda+j)+k+1-|L|}\tcM_L,\]
and so for all $j$ with $j > k - \ell$, we have
\[ m \in \cU^{\lambda +j}\cM \implies m \in \sum_{i=1}^r t_i \cU^{\lambda+j-a_i}\cM.\]

Finally, it is clear that $L(t\de_t) - \lambda + |L|$ is nilpotent on ${\rm Gr}_{\cU}^\lambda(\cM)$, using the relation \eqref{eq-EulerRelated}. 
\end{proof}

Define the bi-filtered Rees module $R_{F,L}(\cM) = \bigoplus_{k,j} F_k {}^L V^j \cM z^k u^{-j}$, which is a module over $R_{F,L}(\cD_T)$.

\begin{lem} \label{lem-EventuallyFiltAcyclic} Consider the $\Z$-indexed ${}^L V$-filtration ${}^L V^\bullet \cM$ on $\cM$. Then $R_{F,L}(\cM)$ is coherent over $R_{F,L}(\cD_T)$.

In particular, for integer $k \gg 0$, the complex $A^k_L(\cM,F)$ is filtered acyclic.
\end{lem}
\begin{proof} The second claim follows immediately from the first using \lemmaref{lem-EventuallyFiltAcyclicGood}.

The first claim follows from the observation that, up to a shift of grading, we have $R_{F,L}(\cM) = R_F(V^0 \widetilde{\cM}_L)$ and the isomorphism $R_F(\cD_{\widetilde{T}^L/\A^1_u}) \cong R_{F,L}(\cD_T)$ from \lemmaref{lem-LArtinRees}. We know by Remark \ref{rmk-goodnessOfFOnV0} that $R_F(V^0\widetilde{\cM}_L)$ is coherent over $R_F(\cD_{\widetilde{T}^L/\A^1_u})$, proving the claim.
\end{proof}

Define ${\rm Sp}_L(M) = \psi_u(\widetilde{M}_L)$, which is a mixed Hodge module on $X\times \A^r_z$, where we use $z_i = \frac{t_i}{u^{a_i}}$ as above. Its underlying filtered $\cD$-module is given by
\[ F_p {\rm Sp}_L(\cM) = \bigoplus_{\chi \in \Q} F_p {\rm Gr}_L^\chi(\cM),\]
which is seen by taking $\bigoplus_{\lambda \in (0,1]} {\rm Gr}_V^{\lambda}(\widetilde{\cM}_L)$ and using the formulas of \propositionref{prop-VFiltLSpec}.

This is an example of an \emph{$L$-monodromic} mixed Hodge module, i.e., one whose underlying $\cD$-module is $L$-monodromic. Recall that this means that every local section $m$ is annihilated by some polynomial in $L(z\de_z) = \sum_{i=1}^r a_i z_i \de_{z_i}$. Such modules decompose into generalized eigenspaces for the operator $L(z\de_z)$: we write 
\[ \cN = \bigoplus_{\chi \in \Q} \cN^\chi,\]
where $\cN^{\chi} = \bigcup_{j\geq 1} \ker((L(z\de_z) - \chi + |L|)^j)$. Any $L$-monodromic module $\cN$ carries a nilpotent $\cD$-linear endomorphism $N$ which acts on $\cN^\chi$ by $L(z\de_z) - \chi + |L|$. 

Note that any sub-object or quotient object of an $L$-monodromic $\cD$-module is also $L$-monodromic. Given a short exact sequence
\[ 0 \to \cN_1 \to \cN_2 \to \cN_3 \to 0\]
of $L$-monodromic $\cD$-modules, for any $\chi \in \Q$, the sequence
\[ 0 \to \cN_1^\chi \to \cN_2^\chi \to \cN_3^\chi \to 0\]
is an exact sequence of ${\rm Gr}_L^0 \cD$-modules.

\begin{eg} Consider $\cO_T = \cO_X[t_1,\dots, t_r]$. This is $L$-monodromic for any slope $L$. Indeed, for any $h\in \cO_X$, we have
\[ L(t\de_t)(h t^\alpha) = L(\alpha) ht^\alpha.\]

Hence, for this fixed $L$, we have
\[ \cO_T^\chi = \bigoplus_{L(\alpha) = \chi - |L|} \cO_X t^\alpha.\]
\end{eg}

We use the fact that the complex $B^0_L(\cM)$ computes the $\cD$-module theoretic restriction. This shows that ${\rm Sp}_L(\cM)$ can be used to compute the restriction. The proposition below is shown following the proof for the usual $V$-filtration \cite{SaitoMHM}*{Pg. 269}:
\begin{prop} \label{prop-SpComputesRestr} Let $M$ be a mixed Hodge module on $T$. There are canonical quasi-isomorphisms in $D^b({\rm MHM}(X))$:
\[ \sigma^!(M) \cong \sigma^! {\rm Sp}_L(M)\]
\[ \sigma^*(M) \cong \sigma^* {\rm Sp}_L(M).\]
\end{prop}
\begin{proof} The first claim follows for underlying $\cD$-modules by \lemmaref{lem-DModuleRestriction}. Indeed, 
\[
\sigma^!(\cM) \cong B_L^0(\cM) = B_L^0({\rm Sp}_L(\cM)) \cong \sigma^! {\rm Sp}_L(\cM),\]
where the complexes are identified because 
\[ {\rm Sp}_L(\cM)^{\chi} = {\rm Gr}_V^{\chi}(\cM),\]
compatibly with the action of $t_1,\dots, t_r$.

Kashiwara's equivalence shows that ${\rm Sp}_L \circ \sigma_* = \sigma_*$, i.e., that specialization is the identity on modules supported on $X\times \{0\}$.

Let $j \colon T \setminus (X\times \{0\}) \to T$ be the inclusion of the complement of the zero section. We have morphisms
\[ j_*j^* {\rm Sp}_L(M) \to j_* j^* {\rm Sp}_L(j_* j^*(M))\]
\[ {\rm Sp}_L(j_* j^*(M)) \to j_* j^* {\rm Sp}_L(j_* j^*(M)).\]

The cones of these morphisms vanish because their underlying complexes of $\cD$-modules do. Thus, these morphisms are quasi-isomorphisms.

Thus, starting with
\[ \sigma_* \sigma^!(M) \to M \to j_* j^*(M) \xrightarrow[]{+1},\]
when we apply ${\rm Sp}_L(-)$, we get
\[\sigma_* \sigma^!(M) \to {\rm Sp}_L(M) \to j_*j^*{\rm Sp}_L(M) \xrightarrow[]{+1},\]
which gives a canonical isomorphism
\[ \sigma_* \sigma^!(M) \cong \sigma_*\sigma^! {\rm Sp}_L(M),\]
by, for example, \cite{DirksRadon}*{Lem. 4.4}. The point is that, although cones are not canonical in triangulated categories, they are for morphisms of the form $M' \to j_* j^*(M')$.

The second claim follows from the first by duality, using that ${\rm Sp}_L(-)$ commutes with duality.
\end{proof}

\begin{rmk} The proof above shows that there are canonical quasi-isomorpisms
\[ j_* j^* {\rm Sp}_L(M) \cong {\rm Sp}_L(j_* j^*(\cM)),\]
\[ j_! j^* {\rm Sp}_L(M) \cong {\rm Sp}_L(j_! j^*(\cM)).\]
    
\end{rmk}

We record the following useful fact about $L$-monodromic mixed Hodge modules, with an important observation in the pure case.

\begin{lem} \label{lem-pureLMonodromic} Assume $M$ is $L$-monodromic on $X\times \A^r_z$. Then there is a canonical isomorphism $M \cong {\rm Sp}_L(M)$.

Therefore, if $M$ is pure, we have $N = 0$ on $M$, i.e., $L(z\de_z)$ acts semi-simply on $\cM$.
\end{lem}
\begin{proof} The second claim follows from the first because the weight filtration on ${\rm Sp}_L(M)$ is the monodromy filtration for $N$ (using the definition in terms of nearby cycles). Hence, purity is equivalent to $N = 0$.

Let $\sigma \colon X \times \{0\} \to X\times \A^r_z$ and $j\colon X \times (\A^r_z \setminus \{0\}) \to X\times \A^r_z$ be the natural morphisms.

For the first claim, we have the exact triangles in $D^b({\rm MHM}(X\times \A^r_z))$, here we identify $X\times \A^r_z$ with the normal bundle of $X\times \A^r_z$ along the zero section:
\[ \sigma_* \sigma^! M \to M \to j_* j^*(M) \xrightarrow[]{+1},\]
\[ \sigma_* \sigma^! {\rm Sp}_L(M) \to {\rm Sp}_L(M) \to j_*j^*({\rm Sp}_L(M)) \xrightarrow[]{+1}.\]

By \propositionref{prop-SpComputesRestr} and its proof, we have a morphism of triangles
\[ \begin{tikzcd} j_*j^*(M)[-1] \ar[r]\ar[d] & \sigma_* \sigma^!(M) \ar[r]\ar[d] & M \ar[r,"+1"] \ar[d] & {}\\ j_*j^*({\rm Sp}_L(M))[-1] \ar[r] & \sigma_* \sigma^!({\rm Sp}_L(M)) \ar[r] & {\rm Sp}_L(M) \ar[r,"+1"] & {}\end{tikzcd},\]
where the first two vertical maps are isomorphisms. By the same argument as in \cite{DirksRadon}*{Lem. 4.4}, the third vertical map is unique and is an isomorphism.
\end{proof}

We end this section by mentioning the existence of the relative monodromy filtration on ${\rm Gr}_L^\lambda(\cM)$ for the nilpotent operator $L(t\de_t) - \lambda + |L|$ and for the induced filtration 
\[
M_\bullet {\rm Gr}_L^\lambda(\cM) = {\rm Gr}_L^{\lambda}(W_\bullet \cM).
\]
We denote this filtration $W_\bullet {\rm Gr}_L^\lambda(\cM)$. 

As above, this allows us to define $W_\bullet B^0_L(\cM)$ and $W_\bullet C^0_L(\cM)$, which will be the weight filtration on the complex $B^0_L(\cM)$ and $C^0_L(\cM)$, respectively.

The following can be shown exactly as in the proof of \cite{CD}*{Lem. 6.2}. It says that there exists a splitting of $M_\bullet$ on ${\rm Gr}^W_k {\rm Gr}_L^{\lambda}(\cM)$ which is functorial in a certain sense.

\begin{lem} \label{lem-canonicalSplitBComplex} There exists a splitting of $M_\bullet {\rm Gr}^W_k {\rm Gr}_L^{\lambda}(\cM)$ which is functorial with respect to each of the morphisms:
\[
\begin{aligned}
    t_i \colon {\rm Gr}^W_k {\rm Gr}_L^{\lambda}(\cM) \to {\rm Gr}^W_k {\rm Gr}_L^{\lambda+a_i}(\cM), \\
\de_{t_i} \colon {\rm Gr}^W_k {\rm Gr}_L^{\lambda}(\cM) \to {\rm Gr}^W_k {\rm Gr}_L^{\lambda-a_i}(\cM),     
\end{aligned} 
\]

In particular, the complexes ${\rm Gr}^W_k B^0_L(\cM)$ and ${\rm Gr}^W_k C^0_L(\cM)$ split into their associated graded pieces for the filtration $M_\bullet$
\end{lem}
\begin{proof}[Proof Sketch] The proof uses an idea of Deligne and we refer to~\cite{CD}*{6.4} for the definition of \emph{Deligne-system}. For every $\lambda\in \Q$, Consider the tuple 
\[
(V^\lambda,W_\bullet,N)=({\rm Gr}^M_\bullet{\rm Gr}_L^\lambda \cM,W_\bullet{\rm Gr}^M_\bullet{\rm Gr}_L^\lambda\cM, L(t\de_t) - \lambda + \lvert L \rvert),
\] 
where $V^\lambda$ is taken in the category of filtered $\cD$-modules. 
Let $Y$ be the operator on ${\rm Gr}^M_\bullet{\rm Gr}_L^\lambda \cM$ such that $Y=i$ on  ${\rm Gr}^M_i{\rm Gr}_L^\lambda \cM$ for every $i$. By definition,
\[
    [Y, N] = -2N, \quad \text{and} \quad YW_i \subset W_i \quad \text{for all }i.
\]
These are called \emph{admissibility conditions}. Moreover, by~\cite{SaitoMHM}*{1.5}, there is a canonical splitting
\[
   {\rm Gr}^M_\bullet{\rm Gr}_L^\lambda \cM \cong \bigoplus_{i\in \mathbf{Z}} {\rm Gr}^M_\bullet{\rm Gr}^W_i{\rm Gr}_L^\lambda \cM
\] 
and hence, the set of the splitting operator for $W_\bullet{\rm Gr}^M_\bullet{\rm Gr}_L^\lambda\cM$ commuting with $Y$ is nonempty. Recall that a splitting operator $H$ for $W_\bullet$ is an operator such that $H$ is semisimple and that $W_i= \bigoplus_{k\leq i} E_k(H)$. Then by~\cite{CD}*{Theorem 6.12}, there is a unique splitting operator $Y'$ commuting with $Y$ of $W_\bullet{\rm Gr}^M_\bullet{\rm Gr}_L^\lambda\cM$ such that $(V^\lambda,W_\bullet,N,Y,Y')$ forms a Deligne-system\footnote{To avoid confusion, we point out that we use here different notation from~\cite{CD}*{Theorem 6.12}, the relative monodromy filtration $M$ is denoted by $W$ there, the induced weight filtration $W$ is denoted by $L$ there, and finally ${\rm Gr}_L$ here means the associated graded of ${}^LV$-filtration}. Then the morphism 
\[
t_i\colon {\rm Gr}^M_\bullet{\rm Gr}_L^\lambda \cM \to {\rm Gr}^M_\bullet{\rm Gr}_L^{\lambda+a_i} \cM
\]
induces a morphism of Deligne-systems between
\[
(V^\lambda,W_\bullet,N,Y,Y') \quad \text{and} \quad (V^{\lambda+a_i},W_\bullet,N,Y,Y')
\] 
It follows from~\cite{CD}*{Corollary 6.13} that $[t_i, Y']=0$. 

Similarly, the map $\de_{t_i} \colon {\rm Gr}^M_\bullet {\rm Gr}_L^{\lambda}(\cM) \to {\rm Gr}^M_\bullet {\rm Gr}_L^{\lambda-a_i}(\cM)(-1)$ also induces a morphism between Deligne-systems 
\[
    (V^\lambda,W_\bullet,N,Y,Y') \quad \text{and} \quad (V^{\lambda-a_i}(-1),W_\bullet,N,Y,Y')
\]
and the same argument as above concludes the proof.
\end{proof}

This lemma will be the key to allow us to reduce to the pure case. Indeed, we have that ${\rm Gr}^M_i {\rm Gr}^W_k B^0_L(\cM) = {\rm Gr}^W_k B^0_L({\rm Gr}^W_i \cM)$ and ${\rm Gr}^M_i {\rm Gr}^W_k C^0_L(\cM) = {\rm Gr}^W_k C^0_L({\rm Gr}^W_i \cM)$.

\section{Cyclic Coverings} \label{sect-CyclicCovers} Let $(a_1,\dots, a_r) \in \Z^r_{\geq 1}$. We consider the map $\pi \colon X \times \A^r_w \to X\times \A^r_t$ which sends $(x,w)$ to $(x,w^a)$. For any $(b_1,\dots, b_r) \in \Z_{\geq 0}^r$, we define the slope $\ell = \sum_{i=1}^r b_i s_i$ and consider $L = \ell * a = \sum_{i=1}^r (a_ib_i) s_i$.

Given $\cM$ a regular holonomic $\cD_{X\times \A^r_t}$-module, the pull-back $\pi^! \cM$ agrees with the $\cO$-module pull-back \cite{HTT}*{Thm. 7.1.1}, and is given by
\[ \pi^! \cM = \bigoplus_{0\leq \beta \leq \underline{a-1}} \cM w^\beta,\]
where $\beta \leq \underline{a-1}$ means $\beta_i \leq a_i -1$ for all $1\leq i\leq r$. 

The $\cD$-module action is given by \cite{HTT}*{Sect. 1.3}
\[ P(mw^\beta) = P(m)w^{\beta} \text{ for all } P \in \cD_X,\]
\[ w_i(m w^\beta) = \begin{cases} m w^{\beta + e_i} & \beta_i < a_i -1\\ t_i m w^{\beta - (a_i-1)e_i} & \beta_i = a_i -1\end{cases},\]
\[ \de_{w_i}(m w^\beta) = \begin{cases} (\beta_i + a_i t_i\de_{t_i})(m) w^{\beta-e_i} & \beta_i > 0 \\ a_i \de_{t_i}(m) w^{\beta +(a_i-1)e_i} & \beta_i = 0\end{cases}.\]

Thus, $w_i \de_{w_i}(mw^\beta) = (\beta_i + a_i t_i \de_{t_i})(m) w^\beta$ and so
\[ \ell(w \de_w)(mw^\beta) = (\ell(\beta) + L(t\de_t))(m) w^\beta.\]
 
The first main result of this section is the following:
\begin{thm} \label{thm-cyclicVFilt} Let $\cM$ be a regular holonomic $\cD_{X\times \A^r_t}$-module which is $\Q$-specializable. For any slope $\ell = \sum_{i=1}^r b_i s_i$, write $L = \ell*a$. Then, for all $\lambda \in \Q$, we have
\[ {}^\ell V^\lambda \pi^! \cM = \bigoplus_{0\leq \beta \leq \underline{a-1}} {}^L V^{\lambda + |L| - \ell(\beta) -|\ell|}\cM w^\beta.\]
\end{thm}

\begin{rmk} We will only apply this in the case where $\cM$ is $L$-monodromic, where it becomes essentially trivial. 
\end{rmk}

Before proving this, we study cyclic coverings of the corresponding deformations to the normal bundle. Let $T = X\times \A^r_t$ and $W = X\times \A^r_w$. We consider
\[ \widetilde{T}^L = X \times \A^r_{z} \times \A^1_u\]
\[ \widetilde{W}^\ell = X\times \A^r_{\zeta} \times \A^1_u,\]
where now we should think of $z_i = \frac{t_i}{u^{a_ib_i}}$ and $\zeta_i = \frac{w_i}{u^{b_i}}$. In particular, the morphism $W \to T$ lifts to a finite morphism
\[ \Pi \colon \widetilde{W}^\ell \to \widetilde{T}^L, \quad (x,\zeta,u) \mapsto (x,\zeta^{a_i},u).\]

By \lemmaref{lem-LArtinRees} above, we have isomorphisms
\[ \cD_{\widetilde{T}^L/\A^1_u} \cong R_L(\cD_T),\]
\[ \cD_{\widetilde{W}^\ell/\A^1_u} \cong R_{\ell}(\cD_W).\]

The functor $\Pi^!$ for relative $\cD$-modules agrees with the $\cO$-module pull-back (by flatness of $\Pi$). The functor $\Pi^!$ need not preserve coherence of relative $\cD$-modules. Indeed, even imposing relative holonomicity, coherence is not preserved, see \cite{RelRH}*{Example 2.4}. However, \cite{RelRH}*{Thm. 2} shows that the pull-back of a \emph{regular} relative holonomic $\cD$-module is regular relative holonomic, in particular, coherent.

Recall that a regular relative holonomic $\cD$-module is defined \cite{RelRHTwistor}*{Def. 2.1} to be any relative $\cD$-module $\cN$ with the property that, for all $\lambda \in \C$ (corresponding to a closed point of $\A^1_u$), with fiber $i_\lambda \colon \widetilde{T}^L_\lambda \hookrightarrow \widetilde{T}^L$, all cohomology modules of $L i_\lambda^*(\cN)$ (the derived $\cO$-module restriction) are regular holonomic $\cD_{\widetilde{T}^L_\lambda}$-modules.

\begin{lem} \label{lem-ReesRegRelHolo} Let $\cM$ be a regular holonomic $\cD_T$-module which is $\Q$-specializable. Then for any $\mu \in [0,1)$, the module $R_{L,\mu}(\cM) = \bigoplus_{k\in \Z} {}^L V^{\mu -k}(\cM) u^k$ is a relative regular holonomic $\cD_{\widetilde{T}^L/\A^1_u}$-module.
\end{lem}
\begin{proof} For $\lambda \in \C$, let $i_\lambda \colon \widetilde{T}^L_\lambda \to \widetilde{T}^L$ be the inclusion of the fiber over $\lambda$. We have
\[ L i_\lambda^*(R_L(\cM)) = [ R_L(\cM) \xrightarrow[]{u-\lambda} R_L(\cM)] \cong \begin{cases} \cM & \lambda \neq 0 \\ {\rm Sp}_L(\cM) & \lambda = 0\end{cases},\]
where in the last isomorphism, we use that ${}^L V^{\mu+k} \cM/ {}^L V^{\mu+k+1}\cM = \bigoplus_{\chi \in [\mu+k,\mu+k+1)} {\rm Gr}_L^\chi(\cM)$ for all $k\in \Z$. 

As $\cM$ and ${\rm Sp}_L(\cM)$ are regular holonomic $\cD_T$-modules, this proves that $R_{L}(\cM)$ is regular relative holonomic.
\end{proof}

\begin{proof}[Proof of \theoremref{thm-cyclicVFilt}] We will show that the $\Q$-indexed filtration
\[ U^\bullet \pi^! \cM = \bigoplus_{0\leq \beta \leq \underline{a-1}} {}^L V^{\bullet+|L|-\ell(\beta) - |\ell|} \cM w^\beta\]
satisfies the defining properties of the ${}^\ell V$-filtration. It is certainly discretely, left-continuously indexed and exhaustive.

It is trivial to see that $w_i U^\bullet \subseteq U^{\bullet+b_i}$,  $\de_{w_i} U^\bullet \subseteq U^{\bullet-b_i}$, and that $(\ell(w\de_w)- \lambda + |\ell|)$ is nilpotent on ${\rm Gr}_U^\lambda(\pi^! \cM)$ using the equality
\[ (\ell(w\de_w) - \lambda +|\ell|)(mw^\beta) = (L(t\de_t) - \lambda + \ell(\beta) + |\ell|)(mw^\beta).\]

The last remaining property is coherence of $U^\chi \pi^! \cM$ over ${}^\ell V^0\cD_{X\times \A^r_w}$ for all $\chi \in \Z$. For this claim,  we can focus on the $\Z$-indexed filtration $U^{\mu+\bullet}\pi^! \cM$ for some $\mu \in [0,1)$ fixed.

By \lemmaref{lem-ReesRegRelHolo}, the module $R_{L,\mu}(\cM)$ is regular relative holonomic on $\widetilde{T}^L$. Hence, $\Pi^!$ applied to it remains regular relative holonomic, in particular, coherent. We conclude by showing that this pull-back is isomorphic to $R_{U,\mu}'(\pi^! \cM) = \bigoplus_{k\in \Z} U^{\mu-k-|L|+|\ell|} \pi^! \cM u^k$, proving the claim. Indeed, by coherence of the graded module over $R_{\ell}(\cD_{\widetilde{W}/\A^1_u})$, we conclude that each graded piece is coherent over the $0$th graded piece of that ring, which is ${}^\ell V^0 \cD_{X\times \A^r_w}$.

Recall that the pull-back is simply the $\cO$-module pull-back. For the cyclic cover $\Pi \colon \widetilde{W}^\ell \to \widetilde{T}^L$, we see that
\[ \Pi^!(R_{L,\mu}(\cM)) = \bigoplus_{ 0 \leq \beta \leq \underline{a-1}} R_{L,\mu}(\cM) \zeta^\beta = \bigoplus_{k\in \Z, 0\leq \beta \leq \underline{a-1}} {}^L V^{\mu-k} \cM u^k \zeta^\beta.\]

On the other hand, we have
\[ R_{U,\mu}'(\pi^!(\cM)) = \bigoplus_{k \in \Z} U^{\mu-k-|L|+|\ell|} \pi^!(\cM) u^k = \bigoplus_{k\in \Z, 0 \leq \beta \leq \underline{a-1}} {}^L V^{\mu-k-\ell(\beta)}\cM w^\beta u^k.\]

We define an isomorphism
\[ \Pi^!(R_{L,\mu}(\cM)) \to R_{U,\mu}'(\pi^!(\cM)), \quad m u^k \zeta^\beta \mapsto m u^{k - \ell(\beta)} w^\beta,\]
and it is easy to check that this is a $\cD_{\widetilde{W}^\ell/\A^1_u}$-linear isomorphism. Note that the morphism is motivated by the equality $u^k \zeta^\beta = u^k (\frac{w_1}{u^{b_1}})^{\beta_1} \dots (\frac{w_r}{u^{b_r}})^{\beta_r} = u^{k-\ell(\beta)} w^\beta$.
\end{proof}

\begin{rmk} \label{rmk-cyclicHypersurfaceVFilt} Let $\ell = e_i$, so that ${}^\ell V^\bullet \pi^!\cM$ is the $V$-filtration along $w_i$.

Then $L = a_i e_i$, and it is not hard to check that
\[ {}^L V^\lambda \cM = V_i^{\frac{\lambda}{a_i}} \cM,\]
where the right hand side is the canonical $V$-filtration along $t_i$. Thus, \theoremref{thm-cyclicVFilt} gives the formula
\[ V_i^\lambda \pi^!(\cM) = \bigoplus_{0\leq \beta \leq \underline{a-1}} V_i^{1 + \frac{\lambda - \beta_i - 1}{a_i}}\cM w^\beta.\]
\end{rmk}

For now, we let $r = 1$. We will apply this case later with $X\times \A^{r-1}_t$ in place of $X$. So we consider $\pi \colon X \times \A^1_w \to X\times \A^1_t$ defined by $(x,w) \mapsto (x,w^a)$ for some $a\in \Z_{\geq 1}$. Let $j_w \colon \{w\neq 0\} \to W = X\times \A^1_w$ and $j_t \colon \{t\neq 0\} \to T = X\times \A^1_t$ be the open embeddings. Define the restriction $\pi\vert_{\{w\neq 0\}}$ to be $\rho \colon \{w\neq0\} \to \{t\neq0\}$, which is finite \'{e}tale.

Let $M$ be a pure Hodge module of weight $d$ on $X\times \A^1_t$ with strict support not contained in $\{t = 0\}$. In general, $\pi^! M$ need not remain pure, but we know that $W_{d-1} \pi^! M = 0$.

\begin{lem} \label{lem-structureWd} We can identify $W_d \pi^!(\cM)$ with the minimal extension $j_{!*}(\rho^!(M\vert_{t\neq 0}))$. In particular, $W_d \pi^!(\cM)$ has no quotients supported on $\{w = 0\}$.
\end{lem}
\begin{proof} We have the inclusion $W_d \pi^!(\cM) \subseteq \pi^!(\cM)$, and $\pi^!(\cM)$ has no sub-modules supported in $\{w=0\}$ by adjunction, using the fact that $\cM$ has no sub-modules supported on $\{t=0\}$.

Thus, $W_d \pi^!(\cM)$ has no sub-modules supported in $\{w=0\}$. By polarizability (hence, self-duality), it also admits no quotient objects supported in $\{w=0\}$, proving the claim.
\end{proof}

Thus, we have $W_d \pi^!(\cM) = \cD_{W} \cdot V^{>0}W_d \pi^!(\cM)$ by \lemmaref{lem-strictSuppDMod}, where $V^\bullet$ is the $V$-filtration along $w$. Note that both $W_d \pi^!(\cM)$ and $\pi^!(\cM)$ restrict to $\rho^!(\cM\vert_{\{t\neq0\}})$ on $\{w\neq 0\}$, and so we have equality (by \lemmaref{lem-PositivePieces})
\[V^{>0} W_d \pi^!(\cM) = V^{>0} \pi^!(\cM) = \bigoplus_{0 \leq b \leq a-1} V^{>1- \frac{b+1}{a}}\cM w^b,\]
where the second equality comes from Remark \ref{rmk-cyclicHypersurfaceVFilt}.

We see from the $\cD_W$-module action that $W_d \pi^!(\cM)$ is a graded sub-module of $\pi^!(\cM)$. Throughout, we let
\[ W_d \pi^!(\cM)^b \subseteq \pi^!(\cM)^b\]
denote the $b$th graded piece, i.e., the coefficient of $w^b$.

\begin{lem} We have
\[ V^0 W_d \pi^!(\cM) = \bigoplus_{0\leq b <  a-1} \left(N^{(b)} \cdot V^{1 - \frac{b+1}{a}}\cM + V^{>1-\frac{b+1}{a}}\cM \right)w^b \oplus V^0 \cM w^{a-1},\]
where $N^{(b)}= (t\de_t - (1-\frac{b+1}{a}) + 1)$ is the nilpotent endomorphism of ${\rm Gr}_V^{1-\frac{b+1}{a}}(\cM)$.
\end{lem}
\begin{proof} This follows from $V^0 W_d \pi^!(\cM) = \de_w V^1 \pi^!(\cM) + V^{>0}\pi^!(\cM)$, which holds by the last claim of \lemmaref{lem-structureWd} and the surjectivity
\[ \de_w \colon {\rm Gr}_V^1(W_d \pi^!(\cM)) \to {\rm Gr}_V^0(W_d \pi^!(\cM))\]
which holds by \lemmaref{lem-strictSuppDMod}. To identify the coefficient of $w^{a-1}$, we use the similar formula $V^0\cM = \de_t V^1\cM + V^{>0}\cM$ which also follows by \lemmaref{lem-strictSuppDMod}.
\end{proof}

\begin{lem} \label{lem-HodgeFiltWd} The inclusion
\[ W_d \pi^!(\cM) \to \pi^!(\cM)\]
induces an equality for all $p\in \Z$ and $\lambda > 0$,
\[ F_p V^\lambda W_d \pi^!(\cM) = F_p V^{\lambda} \pi^!(\cM) = \bigoplus_{0\leq b \leq a-1} F_p V^{1+ \frac{\lambda-b-1}{a}}(\cM) w^b,\]
and for $\lambda = 0$, we get equality for all $p\in \Z$:
\[ F_p V^0 W_d \pi^!(\cM)^{a-1} = F_p V^0 \pi^!(\cM)^{a-1} = F_p V^0(\cM) w^{a-1}.\]
\end{lem}
\begin{proof} We already saw the equality
\begin{equation} \label{eq-HodgeVWd} F_p V^\lambda W_d \pi^!(\cM) = F_p V^{\lambda} \pi^!(\cM) \text{ for all } \lambda > 0\end{equation}
above, using the fact that both restrict to $\rho^!(\cM\vert_{\{t\neq0\}})$ on $\{w\neq 0\}$.

For $\lambda = 0$, we have by Lemmas \ref{lem-structureWd} and \ref{lem-strictSupportHodge} the equality
\[ F_p V^0 W_d \pi^!(\cM) = \de_{w_r}(F_{p-1} V^1 W_d \pi^!(\cM)) + F_p V^{>0}W_d\pi^!(\cM)\]
and by the equality \eqref{eq-HodgeVWd}, the right hand side is equal to
\[ \de_{w_r}(F_{p-1} V^1 \pi^!(\cM)) + F_p V^{>0}\pi^!(\cM).\]

Using the fact that $\pi^!(\cM)$ has no $w_r$-torsion, we know by \lemmaref{lem-strictSupportHodge} that for all $\lambda \geq 0$ that we have
\begin{equation}\label{eq-HodgePositiveV} F_p V^\lambda \pi^!(\cM) = V^\lambda \pi^!(\cM) \cap j_*(F_p \rho^!(\cM\vert_{t_r\neq 0\}})) = \bigoplus_{0\leq b \leq a-1} F_p V^{1+\frac{\lambda-b-1}{a}}(\cM) w^b.\end{equation}

Finally, using that $\cM$ has strict support not contained in $\{t_r =0\}$, we get by \lemmaref{lem-strictSupportHodge} $F_p V^0 \cM = \de_t(F_{p-1} V^1 \cM) + F_p V^{>0}\cM$, which is the coefficient of $w^{a-1}$ in $F_p V^0 W_d \pi^!(\cM)$.
\end{proof}

We define $Q = \pi^!(M)/ W_d \pi^!(M)$, and denote the underlying $\cD$-module by $\cQ$. By Kashiwara's equivalence and \lemmaref{lem-structureWd}, we have $Q = i_* Q_0$ where $i\colon \{w=0\} \to W$ is the closed embedding and the underlying $\cD$-module of $Q_0$ is $\cQ_0 = \ker(w) \subseteq \cQ$. Note that by \exampleref{eg-Kashiwara}, we have $V^0 \cQ = \cQ_0$.

\begin{cor} In the above notation, we have equality
\[\cQ_0 = V^0\cQ = \bigoplus_{0\leq b < a-1} {\rm coker}\left(N^{(b)} \colon {\rm Gr}_V^{1-\frac{b+1}{a}}(\cM) \to {\rm Gr}_V^{1-\frac{b+1}{a}}(\cM)\right).\]

In particular, $\pi^!(\cM)$ is pure if and only if the $V$-filtration of $\cM$ along $t$ has no jumping numbers in $\frac{1}{a} \Z \setminus \Z$, i.e., ${\rm Gr}_V^{\lambda}(\cM) \neq 0$ implies $\lambda \notin \frac{1}{a} \Z \setminus \Z$.
\end{cor}
\begin{proof} We have $V^0 \cQ = V^0\pi^!(\cM)/V^0 W_d \pi^!(\cM)$, so the formula is obvious by the previous lemma and \lemmaref{lem-VFiltStrict}.

The last claim follows because $\pi^!(\cM)$ is pure if and only if $\cQ  =0$ if and only if $N^{(b)}$ is surjective on ${\rm Gr}_V^{1-\frac{b+1}{a}}(\cM)$ for all $0\leq b < a-1$, but since $N^{(b)}$ is nilpotent, this is equivalent to the vanishing of ${\rm Gr}_V^{1-\frac{b+1}{a}}(\cM)$.
\end{proof}

As our goal is to study $\sigma^! M$, it is natural to consider $\pi^! M$ because we have the functoriality: $\sigma^! \circ \pi^! = \sigma^!$. However, to study $\sigma^*$ using the cyclic cover, we would need to understand $\pi^* M$. The way to do this is to use duality.

However, it is not easy to be explicit in this description. In light of \lemmaref{lem-CDSGeneral} below, as far as the Koszul-like complexes are concerned, we need only understand ${\rm Gr}_V^0(\pi^*(M))$ and ${\rm Gr}_V^1(\pi^*(M))$, which is far easier.

By definition, $\pi^*(M) = \mathbf D \pi^!(\mathbf D(M)) = \mathbf D \pi^!(M(d))$, using the fact that $M$ is polarizable of weight $d$. Thus, we have isomorphisms of $\cD_X$-modules (ignoring Tate twists)
\[ {\rm Gr}_V^0(\pi^*(\cM)) \cong \mathbf D({\rm Gr}_V^0(\pi^!(\cM)))\]
\[{\rm Gr}_V^1(\pi^*(\cM)) \cong \mathbf D({\rm Gr}_V^1(\pi^!(\cM))),\]
where we use the compatibility of nearby and vanishing cycles with the dual functor \cite{SaitoDual}*{Thm. 1.6}. 

By Formula \eqref{eq-HodgePositiveV}, we have
\[{\rm Gr}_V^0(\pi^!(\cM),F) = \bigoplus_{0\leq b\leq a-1} {\rm Gr}_V^{1-\frac{b+1}{a}}(\cM,F)w^b,\]
\[{\rm Gr}_V^1(\pi^!(\cM),F) = \bigoplus_{0\leq b\leq a-1} {\rm Gr}_V^{1 - \frac{b}{a}}(\cM,F) w^b.\]

The map ${\rm Var} \colon {\rm Gr}_V^0(\pi^!(\cM),F) \to {\rm Gr}_V^1(\pi^!(\cM),F)$ decomposes along the grading: for $b < a-1$, we have
\begin{equation} \label{eq-VarId}{\rm Var} \colon {\rm Gr}_V^0(\pi^!(\cM),F)^b \to {\rm Gr}_V^1(\pi^!(\cM),F)^{b+1} = \left[{\rm Gr}_V^{1-\frac{b+1}{a}}(\cM,F) \xrightarrow[]{\rm Id} {\rm Gr}_V^{1-\frac{b+1}{a}}(\cM,F)\right]\end{equation}
and for $b = a-1$, we have
\begin{equation} \label{eq-VarVar}{\rm Var} \colon {\rm Gr}_V^0(\pi^!(\cM,F))^{a-1} \to {\rm Gr}_V^1(\pi^!(\cM,F))^{0} = \left[{\rm Gr}_V^{0}(\cM,F) \xrightarrow[]{\rm Var} {\rm Gr}_V^{1}(\cM,F)\right]\end{equation}

Moreover, as ${\rm can}\colon {\rm Gr}_V^1(\pi^*(\cM),F) \to {\rm Gr}_V^0(\pi^*(\cM),F[-1])$ is the dual of the map 
\[ {\rm Var} \colon {\rm Gr}_V^0(\pi^!(\cM)) \to {\rm Gr}_V^1(\pi^!(\cM)), \] the same applies to the can and Var maps for $\cM$. Thus, we have for $b>0$
\begin{equation} \label{eq-canId} {\rm can} \colon {\rm Gr}_V^1(\pi^*(\cM),F)^b \to {\rm Gr}_V^0(\pi^*(\cM),F[-1])^{b-1} = \left[{\rm Gr}_V^{\frac{b}{a}}(\cM,F) \xrightarrow[]{\rm Id} {\rm Gr}_V^{\frac{b}{a}}(\cM,F)\right]\end{equation}
and for $b = 0$, we have
\begin{equation} \label{eq-cancan} {\rm can} \colon {\rm Gr}_V^1(\pi^*(\cM),F)^{0} \to {\rm Gr}_V^0(\pi^*(\cM),F[-1])^{a-1} = \left[{\rm Gr}_V^{1}(\cM,F) \xrightarrow[]{\rm can} {\rm Gr}_V^{0}(\cM,F[-1])\right]\end{equation}

\begin{lem} \label{lem-compareWd} We have natural filtered isomorphisms
\[ {\rm Gr}_V^0(W_d \pi^!(\cM))^{a-1} \to {\rm Gr}_V^0(\pi^!(\cM),F)^{a-1}\]
\[ {\rm Gr}_V^0(\pi^*(\cM),F)^{a-1} \to {\rm Gr}_V^0(\pi^*(\cM)/W_{d-1}\pi^*(\cM),F)^{a-1},\]
and for any $0 \leq b \leq a-1$, we have filtered isomorphisms
\[ {\rm Gr}_V^1(W_d \pi^!(\cM),F)^{b} \to {\rm Gr}_V^1(\pi^!(\cM),F)^{b}\]
\[ {\rm Gr}_V^1(\pi^*(\cM),F)^{b} \to {\rm Gr}_V^1(\pi^*(\cM)/W_{d-1}\pi^*(\cM),F)^{b}.\]
\end{lem}
\begin{proof} All morphisms are induced by taking ${\rm Gr}_V^0$ or ${\rm Gr}_V^1$ of the natural morphisms of filtered $\cD_{X\times \A^1_w}$-modules. The first isomorphism follows from \lemmaref{lem-HodgeFiltWd} and the second follows from the first by duality.

The latter two isomorphisms come from the fact that all three modules restrict on $\{w\neq 0\}$ to $\rho^!(\cM\vert_{t\neq 0})$.
\end{proof}

We can now prove the main result of the paper.

\subsection{Proof of \theoremref{thm-main}}
We use induction on $r$, the case $r=1$ being known. For our fixed $L$, we assume without loss of generality that $a_r > 1$.

We give an outline of the proof for the reader's convenience. First of all, we reduce (by replacing $\cM$ with its $L$-specialization) to the case that $\cM$ is a strict support, pure Hodge module which is $L$-monodromic and whose support is not contained in $\{t_r =0\}$.

Following the strategy of \cite{CDS}, we consider the filtration induced by $V_r^\bullet \cM$ (the $V$-filtration along $t_r$) on the filtered complexes $B^\chi_L(\cM,F)$ and $C^\chi_L(\cM,F)$. In certain ranges, we show that taking $V_r^\beta$ (resp. ${\rm Gr}_{V_r}^\beta(-)$) gives quasi-isomorphic complexes. This leads to a quick proof of the filtered acyclicity statement of \theoremref{thm-main}.

For the statement about restriction functors, we use induction on number of coefficients of $L$ which are strictly larger than $1$ (the base case being that they are all equal to $1$, which is known by \cites{CD,CDS}). We let $\ell = L + (1-a_r)s_r$ be the linear form which agrees with $L$ except has $1$ for its $r$th coefficient. Inductively, we assume that restriction functors can be computed using the ${}^\ell V^\bullet$-filtration.

Then we use the cyclic cover $\pi \colon X \times \A^r_w \to X \times \A^r_t$ defined by 
\[
(x,w_1,\dots, w_r) \mapsto (x,w_1,\dots, w_r^{a_r}),
\]
which relates the ${}^L V^\bullet$ filtration on $\cM$ with the ${}^\ell V^\bullet$-filtration on $\pi^! \cM$. Although we do not have a good understanding of the Hodge filtration on $\pi^!(\cM)$, we do have a good understanding for the Hodge filtration on $V_r^\bullet \pi^!(\cM)$, where $V_r^\bullet$ is the $V$-filtration along $w_r$. In terms of the complexes $B^\chi_\ell(\pi^!(\cM),F)$ and $C^\chi_{\ell}(\pi^!(\cM),F)$, in the range of interest, we can consider the $V_r^\bullet$-version of those complexes. Putting all of this together allows us to prove the claim.

We now begin with the proof.

Note that by \lemmaref{lem-EventuallyFiltAcyclic}, in order to prove filtered acyclity of $A^\chi_L(\cM,F)$ for all $\chi > 0$, it suffices to prove it for $B^\chi_L(\cM,F)$ for all $\chi > 0$. Thus, we only need to consider the complexes $B_L^\bullet(\cM,F)$ and $C_L^\bullet(\cM,F)$ below. 

For any $M$ a mixed Hodge module on $X\times \A^r_t$, by \propositionref{prop-SpComputesRestr} we can replace $M$ by ${\rm Sp}_L(M)$ to assume that $M$ is $L$-monodromic. Indeed, by definition we have $B^\chi_L(\cM,F) = B^\chi_L({\rm Sp}_L(\cM),F)$ and $C^\chi_L(\cM,F) = C^\chi_L({\rm Sp}_L(\cM),F)$ for all $\chi \geq 0$.

Moreover, by the functorial splitting Lemma~\ref{lem-canonicalSplitBComplex} of the relative monodromy filtration, we can assume $M$ is $L$-monodromic and pure of weight $d$. By the strict support decomposition, we can assume $M$ has strict support not contained in $\{t_r = 0\}$. Indeed, if $M$ is supported in $\{t_r =0\}$, the result follows by induction on $r$ using \exampleref{eg-SupportedHypersurface}.

The most important technical input to the proof of \theoremref{thm-main} is the $L$-monodromic analogue of \cite{CDS}*{(2.1.4)}. For this, we let $M$ be an $L$-monodromic module on $T$. Let $V_r^\bullet \cM$ be the $V$-filtration along $t_r$. For any $\beta \in \Q$, the module ${\rm Gr}_{V_r}^\beta(\cM)$ is $L'$-monodromic, where $L ' = \sum_{i=1}^{r-1} a_i s_i$. Indeed, if we write $\cM = \bigoplus_{\chi \in \Q} \cM^\chi$, then because $t_r \de_{t_r}$ and $L(t\de_t)$ commute, we have equality $V_r^\beta \cM= \bigoplus_{\chi \in \Q} V_r^\beta \cM^\chi$, where $V_r^\beta \cM^\chi = \cM^\chi \cap V_r^\beta \cM$. If we write the monodromic decomposition as
\[ {\rm Gr}_{V_r}^\beta(\cM) = \bigoplus_{\gamma \in \Q} {\rm Gr}_{V_r}^\beta(\cM)^\gamma,\]
where by definition ${\rm Gr}_{V_r}^\beta(\cM)^{\gamma} = \bigcup_{j\geq 1} \ker(L'(t\de_t) - \gamma + |L'|)$, then using that $t_r \de_{t_r} - \beta +1$ is nilpotent on ${\rm Gr}_{V_r}^\beta(\cM)$ and that $L(t\de_t) - \chi + |L|$ is nilpotent on $\cM^\chi$, we can identify
\[ {\rm Gr}_{V_r}^\beta(\cM)^{\gamma} = {\rm Gr}_{V_r}^\beta(\cM^{\gamma+a_r \beta}) = \frac{V_r^\beta \cM^{\gamma+a_r \beta}}{V_r^{>\beta} \cM^{\gamma+a_r\beta}}.\]

Now, the filtration $V_r^\bullet$ defines filtrations on $B_L^\lambda(\cM,F)$ and $C_L^\lambda(\cM,F)$. We use the Kronecker delta notation: $\delta_{ir} = \begin{cases} 0 & i \neq r \\ 1 & i = r \end{cases}$. Then the filtration is given by
\[
{\tiny
V_r^\beta B^\lambda_L(\cM) = \left[V_r^\beta (\cM^\lambda,F[-r]) \xrightarrow[]{t_i} \bigoplus_{i=1}^r V_r^{\beta+\delta_{ir}}(\cM^{\lambda+a_i},F[-r]) \xrightarrow[]{t_i} \dots \xrightarrow[]{t_i} V_r^{\beta+1}(\cM^{\lambda+|L|},F[-r])\right]}
\]

\[
\begin{aligned}
&V_r^\beta(C^\lambda_L(\cM),F)= \\
&\left[ V_r^{\beta+1}(\cM^{|L|+\lambda},F) \xrightarrow[]{\de_{t_i}} \bigoplus_{i=1}^r V_r^{\beta+1-\delta_{ir}}(\cM^{|L|+\lambda-a_i},F[-1]) \xrightarrow[]{\de_{t_i}} \dots \xrightarrow[]{\de_{t_i}} V_r^\beta(\cM^{\lambda},F[-r])\right].
\end{aligned}
\]

\normalsize

With this definition and the standard construction showing how Koszul complexes are obtained by mapping cones of multiplication morphisms, we see that ${\rm Gr}_{V_r}^\beta B_L^\lambda(\cM,F)$ is filtered quasi-isomorphic to the mapping cone of the natural morphism
\[ B_{L'}^{\lambda-a_r \beta}({\rm Gr}_{V_r}^\beta(\cM),F) \xrightarrow[]{t_r} B_{L'}^{\lambda - a_r \beta}({\rm Gr}_{V_r}^{\beta+1}(\cM),F)\]
and similarly ${\rm Gr}_{V_r}^\beta C_L^\lambda(\cM,F)$ is filtered quasi-isomorphic to the mapping cone of the natural morphism
\[ C_{L'}^{\lambda - a_r \beta}({\rm Gr}_{V_r}^{\beta+1}(\cM),F) \xrightarrow[]{\de_{t_r}} C_{L'}^{\lambda-a_r\beta}({\rm Gr}_{V_r}^\beta(\cM),F[-1]).\]

By induction on $r$, we see then that
\[ {\rm Gr}_{V_r}^\beta B_L^\lambda(\cM,F)\text{ is filtered acyclic for all } \lambda - a_r \beta > 0,\]
because $B_{L'}^{\lambda-a_r\beta}(-)$ is filtered acyclic in this range. Similarly, 
\[ {\rm Gr}_{V_r}^\beta C_L^\lambda(\cM,F)\text{ is filtered acyclic for all } \lambda - a_r \beta < 0.\]

Moreover, using the filtered isomorphisms \eqref{cond-tFiltIso} (which we recall says $t \colon (V_r^\lambda \cN,F) \to (V_r^{\lambda+1} \cN,F)$ is a filtered isomorphism for any $\lambda > 0$ and any $(\cN,F)$ underlying a mixed Hodge module) and \eqref{cond-dFiltIso} (which says $\de_t \colon ({\rm Gr}_V^{\lambda+1}(\cN,F) \to {\rm Gr}_V^{\lambda}(\cN,F[-1])$ is a filtered isomorphism for all $\lambda < 0$ and any $(\cN,F)$ underlying a mixed Hodge module), we see that $V_r^\beta B_L^\lambda(\cM,F)$ is filtered acyclic for all $\beta > 0$ and ${\rm Gr}_{V_r}^{\beta} C_L^\lambda(\cM,F)$ is filtered acyclic for all $\beta < 0$.

\begin{lem}[\cite{CDS}*{Lem. 2.1}] \label{lem-CDS21} If for some $\lambda \in \Q$ fixed the complex ${\rm Gr}_{V_r}^\beta C^\lambda_L(\cM,F)$ is filtered acyclic for all $\beta \geq \beta_0$, then $V_r^{\beta_0} C_L^\lambda(\cM,F)$ is filtered acyclic.
\end{lem}
\begin{proof} The proof is the same as in \cite{CDS}*{Rmk. 2.1g}, though the algebraic preliminary must be slightly adjusted.

Let $R = V_r^0 {\rm Gr}_\bullet^F {\rm Gr}_L^0(\cD_{X\times \A^r})$ and let $p\in X\times \{0\}$ be an arbitrary point. Note that $V_r^\beta {\rm Gr}_\bullet^F \cM_p$ is finite graded over $V_r^0 {\rm Gr}^F_\bullet(\cD_{X \times \A^r,p})$ by Remark \ref{rmk-goodnessOfFOnV0}. Thus, as the negative degree operators in the latter ring are generated in degree $1$ by $\de_{t_1},\dots, \de_{t_{r-1}}$, we have the existence of some $\alpha_0 \ll 0$ such that
\begin{equation} \label{eq-DifficultAlgebra} V_r^\beta {\rm Gr}_\bullet^F \cM^\alpha_{p}= \sum_{j=1}^{r-1} \de_{t_j} V_r^\beta {\rm Gr}^F_{\bullet-1} \cM^{\alpha+a_j}_{p}\end{equation}
for all $\beta > 0, \alpha - a_r\beta < \alpha_0$. Indeed, first consider the finitely many $0 < \beta_0 <\beta_1 < \dots < \beta_m = 1$ corresponding to the jumping coefficients of the $V_r$-filtration in the interval $(0,1]$. If $m_1^i,\dots, m_{n_i}^i$ are graded generators for $V_r^{\beta_i} {\rm Gr}_{\bullet}\cM_p$ of degrees $\gamma_1^i,\dots, \gamma_{n_i}^i \in \Q$, then for any $\alpha < \gamma^i =\min \{\gamma_1^i,\dots, \gamma_{n_i}^i\}$, we have
\[ V_r^{\beta_i} {\rm Gr}_\bullet^F \cM^\alpha_{p}= \sum_{j=1}^{r-1} \de_{t_j} V_r^{\beta_i} {\rm Gr}^F_{\bullet-1} \cM^{\alpha+a_j}_{p}\]
and so by multiplying by $t_r^k$, using the isomorphism \ref{cond-tFiltIso}, we get
\[ V_r^{\beta_i+k} {\rm Gr}_\bullet^F \cM_p^{\alpha+a_r k} = \sum_{j=1}^{r-1} \de_{t_j} V_r^{\beta_i +k} {\rm Gr}^F_{\bullet-1} \cM_p^{\alpha+a_rk + a_j}\]
for $\alpha - a_r \beta_i < \gamma^i - a_r \beta_i$. Note that this is equivalent to the inequlaity
\[ (\alpha+a_r k) - a_r(\beta_i+k) < \gamma^i -a_r,\]
and so is of the desired form. Finally, if we take $\alpha_0 = \min\{\gamma^1 - a_r \beta_1,\dots, \gamma^m - a_r\beta_m\}$. we get the equality \eqref{eq-DifficultAlgebra}.

Iterating this, for $\beta$ fixed and any $\alpha - a_r \beta < \alpha_0$, we can write for any $N \geq 1$ the equality
\[ V_r^\beta {\rm Gr}_\bullet^F \cM_p^\alpha = \sum_{|\gamma| = N, \gamma \in \N^{r-1}} \de_t^\gamma V_r^\beta {\rm Gr}^F_{\bullet - N} \cM^{\alpha+L(\gamma)}_p.\]

There exists $N \gg 0$ and $m \gg 0$ such that if $|\gamma| \geq N$, then $\de_t^\gamma =\de_t^{\gamma'} \de_t^{\varepsilon}$ where $L(\gamma') = ma_r$. Indeed, we can take $m = a_1 \dots a_{r-1}$ and then we want to choose $N$ so that there must exist some $\gamma_i \geq \frac{m a_r}{a_i}$. Then we take $\gamma' = \frac{m a_r}{a_i} e_i$, which satisfies $L(\gamma') = ma_r$. Of course, for $N = \sum_{i=1}^{r-1} \frac{ma_r}{a_i}$, the claim holds.

Let $N, m$ be as above. Then
 \[ V_r^\beta {\rm Gr}_\bullet^F \cM^\alpha_p = \sum_{L(\gamma) = m a_r} \de_t^\gamma V_r^\beta {\rm Gr}^F_{\bullet - |\gamma|} \cM^{\alpha+m a_r}_{p},\]
 and so we get equality for any $\beta > m$ and $\alpha - \beta \leq \alpha_0$.
 \[ V_r^\beta {\rm Gr}_\bullet^F \cM^\alpha_{p} = \sum_{L(\gamma) = m a_r} (t_r^m \de_t^\gamma) V_r^{\beta - m}{\rm Gr}^F_{\bullet - |\gamma|} \cM^\alpha_p.\]
 
 Note that as $\gamma \in \N^{r-1}$, there are a finite number of $\gamma$ such that $L(\gamma) = m a_r$. Thus, the elements $\eta_\gamma = {\rm Gr}^F_{|\gamma|} (t_r^m \de_t^\gamma)$ define a finitely generated graded ideal $I$ inside $R$. Thus, the $V_r$-filtration induced on $V_r^{\beta} {\rm Gr}^F_\bullet \cM_p^\alpha$ is cofinal with the $I$-adic filtration, for any $\beta$.
 
 The proof is then identical to that of \cite{CDS}*{Lem. 2.1}, we reproduce the proof here for convenience. First, it is important to note that because $V_r^\beta {\rm Gr}^F_\bullet \cM_p$ is coherent over the $\Z$-graded Noetherian commutative ring $V_r^0 {\rm Gr}^F_\bullet \cD_{T,p}$ (using the grading by $L$, so that ${\rm Gr}^F_ 0 t_i$ has degree $a_i$ and ${\rm Gr}^F_{1}\de_{t_i}$ has degree $-a_i$), it is a standard fact that the individual graded pieces $V_r^\beta{\rm Gr}^F_\bullet \cM_p^\chi$ are coherent over the degree $0$ part of the ring $V_r^0 {\rm Gr}^F_\bullet \cD_{T.p}$. This degree zero part is clearly $R = V_r^0 {\rm Gr}^F_\bullet {\rm Gr}^0_L \cD_{T,p}$. Thus, the terms of the complex $V_r^\beta {\rm Gr}^F_\bullet C_L^\lambda(\cM)$ are finite over the ring $R$.

Our goal is to show that $V_r^{\beta_0} {\rm Gr}^F_\bullet C_L^\lambda(\cM)$ is acyclic assuming that ${\rm Gr}_{V_r}^{\beta} {\rm Gr}^F_\bullet C_L^\lambda(\cM)$ is acyclic for all $\beta \geq \beta_0$. By discreteness of the $V_r$-filtration, the assumption implies that for any $\beta_2 \geq \beta_1 \geq \beta_0$, the complex $(V_r^{\beta_1}/V_r^{\beta_2}){\rm Gr}^F_\bullet C_L^\alpha(\cM)$ is acyclic. By the short exact sequence of complexes
\[ 0 \to (V_r^{\beta_1}/V_r^{\beta_2}){\rm Gr}^F_\bullet C_L^\alpha(\cM)_p \to (V_r^{\beta_0}/V_r^{\beta_2}){\rm Gr}^F_\bullet C_L^\alpha(\cM)_p \to (V_r^{\beta_0}/V_r^{\beta_1}) {\rm Gr}^F_\bullet C_L^\alpha(\cM)_p \to 0,\]
we have that the second map is a quasi-isomorphism. Thus, for all $j$, the inverse system $\cH^j(V_r^{\beta_0}/V_r^{\beta_1}) {\rm Gr}^F_\bullet C_L^\alpha(\cM)_p$ as $\beta_1 \to \infty$ trivially satisfies Mittag-Leffler (the transition maps are isomorphisms, hence surjective). By \cite{EGA}*{Prop. 13.2.3}, this implies that the natural morphism
\[ \cH^j \varprojlim_{\beta_1}(V_r^{\beta_0}/V_r^{\beta_1}) {\rm Gr}^F_\bullet C_L^\alpha(\cM)_p \to \varprojlim_{\beta_1} \cH^j(V_r^{\beta_0}/V_r^{\beta_1}) {\rm Gr}^F_\bullet C_L^\alpha(\cM)_p\]
is an isomorphism for all $j$, and as the term on the right is zero, we see that the $V_r$-completion of the complex $V_r^{\beta_0} {\rm Gr}^F_\bullet C_L^\alpha(\cM)_p$ is acyclic.

We have the following general fact: let $R$ be a Noetherian $\N$-graded ring and $I$ an ideal generated by finitely many elements of positive degree. Then $\widehat{M}_I = 0$ if and only if $M = 0$. Indeed, the kernel of the map $M \mapsto \widehat{M}_I$ is $\bigcap_{j \geq 1} I^j M$. Thus, if $\widehat{M}_I = 0$, we get $M = \bigcap_{j\geq 1} I^j M$, which is impossible by the graded Nakayama lemma.

By the discussion above, the completion along the $V_r$-filtration is isomorphic to the completion along the ideal $I$. As the ring $R$ is graded (by the ${\rm Gr}^F_\bullet$ decomposition) and the ideal $I$ is generated by elements of positive degree, this implies that the finite graded $R$-module $\cH^j V_r^{\beta_0} {\rm Gr}^F_\bullet C_L^\alpha(\cM)$ is zero (by the previous paragraph) for all $j$, proving the claim.
\end{proof}

We can now state and prove the main technical tool. Note that this lemma contains the filtered acyclicity statement of \theoremref{thm-main}.

\begin{lem} \label{lem-CDSGeneral} Let $(\cM,F)$ underlie an $L$-monodromic mixed Hodge module on $X\times \A^r_t$. Then
\[ B^\lambda_L(\cM,F) \text{ is filtered acyclic for all } \lambda > 0,\]
\[ C^\lambda_L(\cM,F) \text{ is filtered acyclic for all } \lambda <0,\]
and the natural morphisms
\[{\rm Gr}_{V_r}^0 B_L^0(\cM,F) \leftarrow V_r^0 B_L^0(\cM,F) \to B_L^0(\cM,F),\]
\[{\rm Gr}_{V_r}^0 C_L^0(\cM,F) \leftarrow V_r^0 C_L^0(\cM,F) \to C_L^0(\cM,F)\]
are filtered quasi-isomorphisms.
\end{lem}
\begin{proof} We prove the claims for the $C$-complexes, the $B$-complexes being similar, but using the acyclicity of $V_r^\beta B_L^\lambda(\cM,F)$ for $\beta > 0$ rather than \lemmaref{lem-CDS21}.

Above, we noted that ${\rm Gr}_{V_r}^\beta C_L^\lambda(\cM,F)$ is filtered acyclic for $\beta < 0$ or for $\lambda < a_r \beta$. First, if $\lambda < 0$, the second case shows that ${\rm Gr}_{V_r}^\beta C_L^\lambda(\cM,F)$ is acyclic for all $\beta \in \Q$. But then by \lemmaref{lem-CDS21}, we see that $V_r^\beta C_L^\lambda(\cM,F)$ is filtered acyclic for all $\beta$, so taking direct limits as $\beta \to -\infty$ gives the claim.

For $\lambda = 0$, we see that
\[ {\rm Gr}_{V_r}^\beta C_L^0(\cM,F) \text{ is filtered acyclic for all } \beta \neq 0,\]
and if $\beta = \varepsilon$ is so that $V_r^{\varepsilon}\cM = V_r^{>0}\cM$, \lemmaref{lem-CDS21} shows that $V_r^{\varepsilon} C_L^0(\cM,F) = V_r^{>0}C_L^0(\cM,F)$ is filtered acyclic. This proves that the arrow
\[ {\rm Gr}_{V_r}^0 C_L^0(\cM,F) \leftarrow V_r^0 C_L^0(\cM,F) \text{ is a filtered quasi-isomorphism}.\]

But then using that $V_r^\beta C_L^0(\cM,F)$ is filtered acyclic for all $\beta < 0$, we see that the natural inclusion
\[ V_r^0 C_L^0(\cM,F) \to V_r^\beta C_L^0(\cM,F)\]
is a filtered quasi-isomorphism for all $\beta < 0$, and taking the direct limit as $\beta \to -\infty$ gives the claim.
\end{proof}

We will prove the (bi-filtered) quasi-isomorphisms \[B^0_L(\cM,F,W) \cong \sigma^!(\cM,F,W)\] and \[C^0_L(\cM,F,W) \cong \sigma^*(\cM,F,W)\] by induction on $|I_L|$ where $I_L = \{i \mid a_i > 1\}$. The case $|I_L| = 0$ is \cite{CD}*{Thm. 1, 2}. So assume by induction that we can compute restriction functors for any $\ell = \sum_{i=1}^r b_i s_i$ with $|I_\ell| < |I_L|$. Recall that we assume (without loss of generality) that $a_r = a > 1$.

Let $\pi \colon W = X\times \A^r_w \to X\times \A^r_t = T$ be the map $(x,w) \mapsto (x,w_1,w_2,\dots, w_r^a)$. We can use the results of the above sub-section for this cyclic cover, as only one variable is being raised to a power.

Let $\sigma_w \colon X \to W$ and $\sigma_t \colon X \to T$ be the inclusion of the zero section. Then $\pi \circ \sigma_w = \sigma_t$. Thus, we have natural isomorphisms
\[ \sigma_w^!(\pi^! M) \cong \sigma_t^!(M),\]
\[ \sigma_w^*(\pi^* M) \cong \sigma_t^*(M).\]

Note that $\pi^!(M)$ is $\ell$-monodromic, where $\ell = s_r + \sum_{i=1}^{r-1} a_i s_i$, as we can easily see from the definition of the $\cD$-module action.

Let $V_r^\bullet\cM$ and $V_r^\bullet\pi^!(\cM)$ be the $V$-filtration along $t_r$ (resp. $w_r$). By \lemmaref{lem-CDSGeneral} we have filtered quasi-isomorphisms
\[ {\rm Gr}_{V_r}^0 B_L^0(\cM,F) \cong B_L^0(\cM,F), \quad \text{and}\quad {\rm Gr}_{V_r}^0 B_\ell^0(\pi^!(\cM),F) \cong B_\ell^0(\pi^!(\cM),F),\]
and we have filtered quasi-isomorphisms
\[ {\rm Gr}_{V_r}^0 C_L^0(\cM,F) \cong C_L^0(\cM,F), \quad \text{and}\quad {\rm Gr}_{V_r}^0 C_\ell^0(\pi^*(\cM),F) \cong C_\ell^0(\pi^*(\cM),F).\]

The filtered quasi-isomorphisms of the $B_\ell$ and $C_\ell$ complexes are $\Z/a\Z$-graded because the Koszul differentials preserve the grading (up to a shift). Here we grade the complexes so that the $b$th graded piece corresponds to the coefficient of $w^b$ in the left-most term.

\begin{lem} \label{lem-GradedPiece} For any $b < a-1$ the complexes
\[{\rm Gr}_{V_r}^0 B_\ell^0(\pi^!(\cM),F)^b \cong B_\ell^0(\pi^!(\cM),F)^b\]
are filtered acyclic. Similarly, for $b > 0$, the complexes
\[{\rm Gr}_{V_r}^0 C_\ell^0(\pi^*(\cM),F)^b \cong C_\ell^0(\pi^*(\cM),F)^b\]
are filtered acyclic.

Thus, there are filtered quasi-isomorphisms:
\[{\rm Gr}_{V_r}^0 B_\ell^0(\pi^!(\cM),F)^{a-1} \cong B_\ell^0(\pi^!(\cM),F)^{a-1}\cong{\rm Gr}_{V_r}^0 B_\ell^0(\pi^!(\cM),F) \cong B_\ell^0(\pi^!(\cM),F)\]
\[{\rm Gr}_{V_r}^0 C_\ell^0(\pi^*(\cM),F)^0 \cong C_\ell^0(\pi^*(\cM),F)^0 \cong {\rm Gr}_{V_r}^0 C_\ell^0(\pi^*(\cM),F) \cong C_\ell^0(\pi^*(\cM),F)\]
\end{lem}
\begin{proof} The last claim is immediate from the filtered acyclicity claims. The filtered acyclicity is immediate in the cases mentioned because one of the morphisms in the Koszul complex will be the identity by the formulas \ref{eq-VarId} and \ref{eq-canId}.
\end{proof}

\begin{proof}[Proof of \theoremref{thm-main}, Restriction Functor]
We begin by proving the claim for the Hodge filtration.

By the inductive hypothesis, we know that the $B_\ell^0$ complex is strictly filtered and computes $\sigma_w^!$ and that $C_\ell^0$ is strictly filtered and computes $\sigma_w^*$. Thus, by \lemmaref{lem-GradedPiece}, we have filtered quasi-isomorphisms
\[ \sigma_t^!(\cM,F) \cong \sigma_w^!\pi^!(\cM,F) \cong {\rm Gr}_{V_r}^0 B_\ell^0(\pi^!(\cM),F)^{a-1} \]
\[ \sigma_t^*(\cM,F) \cong \sigma_w^*\pi^*(\cM,F) \cong {\rm Gr}_{V_r}^0 C_\ell^0(\pi^*(\cM),F)^0.\]

But then it follows from the formula for the Hodge pieces on $V^\lambda \pi^!(\cM)$ in \lemmaref{lem-HodgeFiltWd} for $\lambda \geq 0$ and Equation \ref{eq-VarVar} that we have a filtered quasi-isomorphism
\[{\rm Gr}_{V_r}^0 B_L^0(\cM,F) \cong {\rm Gr}_{V_r}^0 B_\ell^0(\pi^!(\cM),F)^{a-1}.\]

Finally, by Equation \eqref{eq-cancan} and the computation of ${\rm Gr}_V^0(\pi^*(\cM)), {\rm Gr}_V^1(\pi^*(\cM))$ in terms of those on $\pi^!(\cM)$, we have a filtered quasi-isomorphism
\[{\rm Gr}_{V_r}^0 C_L^0(\cM,F) \cong {\rm Gr}_{V_r}^0 C_\ell^0(\pi^*(\cM),F)^0.\]

By one more application of \lemmaref{lem-CDSGeneral}, we obtain the filtered quasi-isomorphisms

\[B_L^0(\cM,F) \cong {\rm Gr}_{V_r}^0 B_L^0(\cM,F) \cong {\rm Gr}_{V_r}^0 B_\ell^0(\pi^!(\cM),F)^{a-1} \cong \sigma_t^!(\cM,F),\]
\[C_L^0(\cM,F) \cong {\rm Gr}_{V_r}^0 C_L^0(\cM,F) \cong {\rm Gr}_{V_r}^0 C_\ell^0(\pi^*(\cM),F)^0 \cong \sigma_t^*(\cM,F),\]
proving the claim for the Hodge filtration.

For the weight filtration, as $M$ is pure $L$-monodromic, we know that the nilpotent endomorphism $N \colon M \to M$ is $0$ by \lemmaref{lem-pureLMonodromic}. Thus, the same is true for $\pi^!(M)$, though $\pi^!(M)$ need not be pure. 

Thus, for any $\chi \in \Z$, the relative monodromy filtration on ${\rm Gr}_\ell^\chi(\pi^!(\cM))$ is simply the induced filtration from $W_\bullet \pi^!(\cM)$, i.e.,
\[ W_\bullet{\rm Gr}_\ell^\chi(\pi^!(\cM)) = {\rm Gr}_\ell^\chi(W_\bullet \pi^!(\cM)),\]
and so we have
\[ W_\bullet B^0_\ell(\pi^!(\cM),F) = B^0_\ell(W_\bullet \pi^!(\cM),F).\]

By \lemmaref{lem-CDSGeneral} applied to $W_d \pi^!(M)$ and $\pi^*(M)/W_{d-1} \pi^*(M)$, we have $\Z/a\Z$-graded filtered isomorphisms
\[B_\ell^0(W_d \pi^!(\cM),F) \cong {\rm Gr}_{V_r}^0 B_\ell^0(W_d\pi^!(\cM),F),\]
\[C_\ell^0(\pi^*(\cM)/W_{d-1}\pi^*(\cM),F) \cong {\rm Gr}_{V_r}^0 C_\ell^0(\pi^*(\cM)/W_{d-1}\pi^*(\cM),F).\]

Now, by \ref{lem-compareWd}, we have natural filtered quasi-isomorphisms
\[ {\rm Gr}_{V_r}^0 B_\ell^0(\pi^!(\cM),F)^{a-1} \cong {\rm Gr}_{V_r}^0 B_\ell^0(\pi^!(\cM),F)^{a-1}\]
\[ {\rm Gr}_{V_r}^0 C_\ell^0(\pi^*(\cM),F)^0 \cong  {\rm Gr}_{V_r}^0 C_\ell^0(\pi^*(\cM)/W_{d-1}\pi^*(\cM),F)^0\]
which finishes the proof.
\end{proof}

As a corollary, we obtain the following analogue of \cite{Ginsburg}*{Prop. 10.4}.

\begin{cor} Let $\cM$ be an $L$-monodromic regular holonomic $\cD_T$-module. Let $\sigma \colon X \to T$ be the zero section and $p\colon T\to X$  the projection. Then the natural morphisms
\[ p_* \cM \to \sigma^* \cM, \ \ \sigma^! \cM \to p_! \cM\]
are isomorphisms in $D^b_{\rm rh}(\cD_X)$.

If $M$ is an $L$-monodromic mixed Hodge module, then the natural morphisms
\[ p_* M \to \sigma^* M,  \ \ \sigma^! M \to p_! M\]
are isomorphisms in $D^b({\rm MHM}(X))$.
\end{cor}
\begin{proof} The natural morphisms described in the corollary come from the following: let $j\colon T \setminus (X\times \{0\}) \to T$ be the inclusion of the complement of the zero section. Then we have exact triangles
\[ \sigma_* \sigma^!M \to M \to j_*j^*(M) \xrightarrow[]{+1},\]
\[ j_!j^*(M) \to M \to \sigma_* \sigma^*(M) \xrightarrow[]{+1}.\]

If we apply $p_!$ to the first triangle and $p_*$ to the second triangle, then using that $p\circ \sigma = {\rm id}_X$, we have
\[ \sigma^!M \to p_!M \to p_!j_*j^*(M) \xrightarrow[]{+1},\]
\[ p_*j_!j^*(M) \to p_*M \to \sigma^*(M) \xrightarrow[]{+1}.\]

The claim for $p_!$ and $\sigma^!$ follows from the claim for $p_*$ and $\sigma^*$ by duality. As the functor $D^b({\rm MHM}(X)) \to D^b_{\rm rh}(\cD_X)$ is conservative, we need only prove the claim for $\cD_T$-modules.

By definition, $p_* \cM$ is computed by appyling $Rp_*(-)$ to the relative de Rham complex for the projection $T\to X$. As $p$ is affine, $Rp_*(-) =p_*(-)$ is an exact functor, and the relative de Rham complex is the Koszul complex on $\de_{t_1},\dots, \de_{t_r}$. This de Rham complex can be identified with $\bigoplus_{\chi \in \Q} C^\chi_L(\cM)$, and by the acyclicity results of \lemmaref{lem-DModuleRestriction}, we get
\[ p_* \cM \cong \bigoplus_{\chi \in \Q} C^\chi_L(\cM) \cong C^0_L(\cM) \cong \sigma^*(\cM),\]
as desired.
\end{proof}

\section{Singularities of Subvarieties} \label{sect-singularities}
Let $Z = V(f_1,\dots, f_r) \subseteq X$ be a closed subvariety of the smooth variety $X$. Consider
\[ 
\cB_f = i_{f,+}(\cO_X)  = \bigoplus_{\alpha \in \N^r} \cO_X \de_t^\alpha \delta_f, 
\]
a regular holonomic $\cD_T$-module. The $\cD$-module action is given by
\[
\begin{aligned}
h(g \de_t^\alpha \delta_f) & = hg \de_t^\alpha \delta_f \quad \text{for all} \quad  h\in \cO_X, \\
D(g \de_t^\alpha \delta_f) &= D(g) \de_t^\alpha \delta_f - \sum_{i=1}^r D(f_i) g \de_t^{\alpha+e_i}\delta_f \quad \text{for all} \quad D\in {\rm Der}_\C(\cO_X), \\
t_i(g \de_t^\alpha \delta_f) &= f_i g \de_t^\alpha\delta_f - \alpha_i g \de_t^{\alpha-e_i}\delta_f, \\
 \de_{t_i}(g\de_t^{\alpha}\delta_f) &= g \de_t^{\alpha+e_i}\delta_f.
\end{aligned}
\]

The Hodge filtration is given by
\[ F_{p+r} \cB_f = \bigoplus_{|\alpha| \leq p} \cO_X \de_t^\alpha \delta_f,\]
where the shift by $r$ is due to the relative dimension of the graph embedding and the fact that we use left $\cD$-modules.

For any slope $L = \sum_{i=1}^r a_i s_i$, define a $\Z$-indexed filtration on $\cB_f$ by
\[ {}^L G^\bullet(\cB_f) = {}^L V^\bullet \cD_T \cdot \delta_f.\]

Define the \emph{$b_L$-function of $f_1,\dots,f_r$} to be the monic minimal polynomial of the action of $L(s) = \sum_{i=1}^r -a_i \de_{t_i} t_i$ on ${\rm Gr}_{{}^L G}^0(\cB_f)$.

\begin{lem} \label{lem-otherGr} For any $j\in \Z$, we have
\[ b_{L}(L(s)+j) {\rm Gr}_{{}^L G}^j(\cB_f) = 0.\]
\end{lem}
\begin{proof} As
\[ {}^L V^j \cD_T = \sum_{L(\beta) \geq L(\gamma)+j} {}^L V^0 \cD_T \cdot t^\beta \de_t^\gamma,\]
we have a surjection
\[ \bigoplus_{L(\beta) \geq L(\gamma)+j} {}^L G^0 \cB_f \xrightarrow[]{(t^\beta \de_t^\gamma)} {}^L G^j(\cB_f),\]
and if we compose with the projection to ${\rm Gr}_{{}^L G}^j(\cB_f)$, we get a surjection
\[\bigoplus_{L(\beta) = L(\gamma)+j} {}^L G^0 \cB_f \xrightarrow[]{(t^\beta \de_t^\gamma)} {\rm Gr}_{{}^L G}^j(\cB_f),\]
where we can take $=$ in the index set, as any terms with strict inequality necessarily map to $0$ by definition. Finally, for any fixed $\beta,\gamma$, note that ${}^L G^1 \cB_f$ maps to $0$ in the associated graded piece, so we have a surjection
\[\bigoplus_{L(\beta) = L(\gamma)+j} {\rm Gr}_{{}^L G}^0(\cB_f) \xrightarrow[]{\Phi} {\rm Gr}_{{}^L G}^j(\cB_f).\]

As $t^\beta \de_t^\gamma L(s) = (L(s+\beta-\gamma)) t^\beta \de_t^\gamma$, we see that $\Phi \circ L(s) = (L(s)+j)\Phi$, which proves the claim.
\end{proof}

\begin{eg} \label{eg-LBFunctionHypersurface} Let $L = a s_i$ for some $a \in \Z_{>0}$. Let $b_{f_i}(s) = \prod (s+\gamma)^{m_\gamma}$ be the usual Bernstein-Sato polynomial of the hypersurface defined by $f_i$. Then it is easy to check that
\[ b_{L,f}(s) = \prod (s+a\gamma)^{m_\gamma} = a^{\deg b_{f_i}} b_{f_i}\left(\frac{s}{a}\right).\]
\end{eg}

The $b_L$-function satisfies the following Thom-Sebastiani type property, similar to \cite{BMS}*{Thm. 5}. The proof is essentially the same, but we repeat it for convenience.

\begin{prop}\label{prop-LBfunctionTS} Let $f_1,\dots, f_r \in \cO_X(X),\, g_1,\dots, g_c \in \cO_Y(Y)$ for $X,Y$ two smooth complex algebraic varieties. Write $L_1 = \sum_{i=1}^r a_i s_i$ and $L_2 = \sum_{i=1}^c b_i s_i$ and let $L = \sum_{i=1}^r a_i s_i + \sum_{j=1}^c b_j s_{r+j}$.

Let $b_{L_1,f}(w)$ be the $b_{L_1}$-function for $f_1,\dots, f_r$ and define similarly $b_{L_2,g}(w)$ and $b_{L,(f,g)}(w)$. Write
\[ b_{L_1,f}(w) = \prod_\alpha (w+\alpha)^{m_\alpha^{(f)}}, \quad b_{L_2,g}(w) = \prod_{\beta} (w+\beta)^{m_\beta^{(g)}}.\]

Then
\[ b_{L,(f,g)}(w) = \prod (w+\gamma)^{m_\gamma},\]
where $m_\gamma = \max \{ m^{(f)}_\alpha + m^{(g)}_\beta -1 \mid m_\alpha^{(f)},m_\beta^{(g)}>0, \alpha+\beta = \gamma\}$.
\end{prop}
\begin{proof} Let $i_{f,+}(\cO_X) = \cB_f, i_{g,+}(\cO_Y) = \cB_g$ and $i_{(f,g),+}(\cO_{X\times Y}) = \cB_{(f,g)}$. Then, as in the proof of \cite{BMS}*{Thm. 5} we have an isomorphism
\[ \cB_{(f,g)} \cong \cB_f \boxtimes \cB_g.\]

Moreover, the ${}^L G^\bullet$-filtration on the left is given by the convolution of the filtrations ${}^{L_1} G^\bullet \cB_f$ and ${}^{L_2} G^\bullet \cB_g$. In other words,
\[ {}^L G^k \cB_{(f,g)} = \sum_{i+j =k} {}^{L_1} G^i \cB_f \boxtimes {}^{L_2} G^j \cB_g.\]

As in the proof of \cite{BMS}*{Thm. 5}, we have
\[ {\rm Gr}_{{}^L G}^k \cB_{(f,g)} \cong \bigoplus_{i+j =k} {\rm Gr}_{{}^{L_1} G}^i \cB_f \boxtimes {\rm Gr}_{{}^{L_2} G}^j \cB_g.\]

Let $b'(w) = \prod (w+\gamma)^{m_\gamma}$ as defined in the proposition statement. 

By \lemmaref{lem-otherGr}, we see that $b'(L(s)+i+j)$ annihilates ${\rm Gr}_{{}^{L_1} G}^i(\cB_f) \boxtimes {\rm Gr}_{{}^{L_2} G}^j(\cB_g)$ for any $i,j$. We see then that $b_{L,(f,g)}(w) \mid b'(w)$.

On the other hand, by the binomial theorem we see that $b'(w)$ is the minimal polynomial of the action of $L(s)$ on ${\rm Gr}_{{}^{L_1} G}^0(\cB_f) \boxtimes {\rm Gr}_{{}^{L_2} G}^0(\cB_g)$. Thus, as $b_{L,(f,g)}(L(s))$ annihilates this term, we get the other divisibility.
\end{proof}

Next, we review the definitions of higher Du Bois and higher rational singularities.

Given any complex algebraic variety $Z$ of pure dimension $d_Z$ and any $0\leq p \leq d_Z$, we have the $p$th Du Bois complex $\underline{\Omega}_Z^p \in D^b_{\rm coh}(\cO_Z)$, with comparison morphisms $\alpha_p \colon \Omega_Z^p \to \underline{\Omega}_Z^p$, where $\Omega_Z^p$ is the sheaf of $p$-th K\"{a}hler differentials on $Z$. These morphisms are quasi-isomorphisms when $Z$ is smooth.

If $Z$ has local complete intersection singularities, then following \cites{JKSY,MOPW,FL}, we say $Z$ has \emph{$k$-Du Bois singularities} if $\alpha_p$ is a quasi-isomorphism for all $p \leq k$.

Using a resolution of singularities, one can define a morphism $\underline{\Omega}_Z^p \xrightarrow[]{\gamma_p} \mathbf D_Z(\underline{\Omega}_Z^{d_Z - p})$ in $D^b_{\rm coh}(\cO_Z)$. Here $\mathbf D_Z(-)$ is the shifted Grothendieck duality functor $R \cH om(-,\omega_Z^\bullet)[-d_Z]$, where $\omega_Z^\bullet$ is the dualizing complex of $Z$. If $Z$ is smooth, this map is the natural isomorphism $\Omega_Z^p \cong \cH om_{\cO_Z}(\Omega_Z^{d_Z -p},\omega_Z)$.

Assuming still that $Z$ has local complete intersection singularities, we say that $Z$ has $k$-rational singularities if the composition $\Omega_Z^p \to \underline{\Omega}_Z^p \to \mathbf D_Z(\underline{\Omega}_Z^{d_Z -p})$ is a quasi-isomorphism for all $p\leq k$. When $Z$ is smooth, this is the usual isomorphism $\Omega_Z^p \cong \cH om_{\cO_Z}(\Omega_Z^{d_Z-p},\omega_Z)$ of locally free sheaves.

Now, let $Z = V(f_1,\dots, f_r) \subseteq X$ be a local complete intersection subvariety of a smooth variety $X$. Associated to this, we have the \emph{local cohomology mixed Hodge module}
\[ \cH^r_Z(\Q_X^H[\dim X]),\]
with underlying bi-filtered $\cD_X$-module denoted $(\cH^r_Z(\cO_X),F,W)$. 

The standard description of the local cohomology module is as follows: let $\cO_X[\frac{f_i}{f_1\dots f_r}]$ be the localization of $\cO_X$ at all $f_j$ with $j\neq i$ and let $\cO_X[\frac{1}{f_1\dots f_r}]$ be the localization at all $f_i$. These modules naturally underlie mixed Hodge modules on $X$, and we have
\[ \cH^r_Z(\cO_X) = {\rm coker}\left( \bigoplus_{i=1}^r \cO_X\left[\frac{f_i}{f_1\dots f_r}\right] \to \cO_X\left[\frac{1}{f_1\dots f_r}\right]\right).\]

This carries the \emph{pole-order filtration}, defined by
\[ P_k \cH^r_Z(\cO_X) = \{m \in \cH^r_Z(\cO_X) \mid (f_1,\dots, f_r)^{k+1} m = 0\}.\]

It is not hard to see that $F_k \cH^r_Z(\cO_X) \subseteq P_k \cH^r_Z(\cO_X)$, see \cite{MustataPopaDuBois}*{Prop. 7.1}.

Our starting point is the following:
\begin{thm}[\cites{CDMO,CDM1,MustataPopaDuBois}] Let $(\cH^r_Z(\cO_X),F,W)$ be the local cohomology bi-filtered $\cD$-module. Then
\[\widetilde{\alpha}(Z) \geq r+k \iff F_k \cH^r_Z(\cO_X) = P_k \cH^r_Z(\cO_X) \iff Z \text{ has } k\text{-Du Bois singularities},\]
\[ \widetilde{\alpha}(Z) > r+k \iff F_kW_{n+r} \cH^r_Z(\cO_X) = P_k \cH^r_Z(\cO_X) \iff Z \text{ has } k\text{-rational singularities}.\]
\end{thm}

In other words, the structure of the local cohomology mixed Hodge module allows us to give lower bounds on the minimal exponent of a local complete intersection, which controls these classes of singularities.

By \theoremref{thm-main}, we have the following:
\begin{cor} \label{cor-BifilteredIso} Let $(\cH^r_Z(\cO_X),F,W)$ be the local cohomology bi-filtered $\cD$-module. Then for any non-degenerate slope $L$, we have
\[ F_p \cH^r_Z(\cO_X) = \left\{\sum_{|\alpha| \leq p} \frac{\alpha! h_\alpha}{f_1^{\alpha_1+1}\dots f_r^{\alpha_r+1}} \mid \sum_{|\alpha| \leq p} h_\alpha \de_t^\alpha \delta_f \in {}^L V^{|L|}\cB_f\right\}\]
and for any $\ell \in \Z_{\geq 0}$, we have
\[ W_{n+r+\ell}\cH^r_Z(\cO_X) = \left\{\sum_{\alpha} \frac{\alpha! h_\alpha}{f_1^{\alpha_1+1}\dots f_r^{\alpha_r+1}} \mid L(t\de_t)^{\ell+1}\left(\sum_{\alpha} h_\alpha \de_t^\alpha \delta_f\right) \in {}^L V^{>|L|}\cB_f\right\}.\]
\end{cor}

To see this, we must make explicit the isomorphism
\[ {}^L V^{|L|}\cB_f \bigg{/} \sum_{i=1}^r t_i {}^L V^{|L|-a_i}\cB_f = \cH^r A^0(\cB_f) = \cH^r\sigma^! \cB_f \cong \cH^r_Z(\cO_X).\]

Note that, by \cite{CDMO}*{Lem. 5.1}, if we give any endomorphism between these $\cD$-modules, then any other one differs from this by a non-zero scalar multiple. So in terms of the Hodge and weight filtrations, any $\cD$-module isomorphism we define will (up to scalar multiple) agree with the formal bi-filtered isomorphism given by \theoremref{thm-main}. We denote that isomorphism by $\rho \colon {}^L V^{|L|}\cB_f \big{/} \sum_{i=1}^r t_i {}^L V^{|L|-a_i}\cB_f \to \cH^r_Z(\cO_X)$.

In \emph{loc. cit.}, a $\cD_X$-linear morphism $\tau \colon \cB_f \to \cH^r_Z(\cO_X)$ is defined as
\[ \sum_\beta h_\beta \de_t^\beta \delta_f \mapsto \sum_\beta \frac{\beta! h_\beta}{f_1^{\beta_1+1}\dots f_r^{\beta_r+1}}\]
and it is observed that $\sum_{i=1}^r t_i \cB_f \subseteq \ker(\tau)$. In particular, $\sum_{i=1}^r (s_i+1) \cB_f \subseteq \ker(\tau)$.

We define $\tau\colon {}^L V^{|L|}\cB_f \to \cH^r_Z(\cO_X)$ by applying $\tau$ to ${}^L V^{|L|}\cB_f \subseteq \cB_f$, which vanishes on $\sum_{i=1}^r t_i {}^L V^{|L|-a_i}\cB_f \subseteq \sum_{i=1}^r t_i \cB_f$. Thus, we get an induced morphism
\[ \overline{\tau} \colon {}^L V^{|L|}\cB_f \bigg{/} \sum_{i=1}^r t_i {}^L V^{|L|-a_i}\cB_f \to \cH^r_Z(\cO_X).\]

Using the same argument as that in \emph{loc. cit.}, we have the following:

\begin{lem} The map $\overline{\tau}$ is surjective.
\end{lem}
\begin{proof} Let $u = \frac{g}{(f_1\dots f_r)^m} \in \cO_X[\frac{1}{f_1\dots f_r}]$ with $m\geq 1$, by definition, we have $u = \tau(v)$ where $v = \frac{g}{(m-1)!^r}(\de_{t_1}\dots \de_{t_r})^{m-1} \delta_f \in \cB_f$. However, $v$ needs not lie in ${}^L V^{|L|}\cB_f$. If it does, we are done.

Otherwise, by discreteness of the ${}^L V$-filtration and nilpotency of $L(s)+\lambda$ on ${\rm Gr}_L^{\lambda}(\cB_f)$, we can find $\alpha_1 \leq \dots \leq \alpha_N < |L|$ such that
\[ (L(s)+\alpha_1)\dots (L(s)+\alpha_N) v \in {}^L V^{|L|}\cB_f.\]

By B\'{e}zout's relation, we have some $p(w), q(w)$ such that
\[ (w+|L|)p(w) + q(w)\prod_{i=1}^N (w+\alpha_i) = 1 \in \C[w].\]

Plugging in $L(s)$ and applying to $v$, we get
\[ (L(s)+|L|)p(L(s))v + q(L(s))\prod_{i=1}^N (L(s)+\alpha_i) v = v.\]

Note that we have $(L(s)+ |L|) = \sum_{i=1}^r a_i (s_i +1)$.

As $(L(s)+ |L|) p(L(s)) v = \sum_{i=1}^r (s_i +1) a_i p(L(s)) v \in \sum_{i=1}^r (s_i+1) \cB_f \subseteq \ker(\tau)$, we conclude.
\end{proof}

\begin{proof}[Proof of \theoremref{thm-LVDescribeLocalCohomology}] We have shown above that $\overline{\tau} \circ \rho^{-1} \in {\rm End}(\cH^r_Z(\cO_X))$ is surjective, hence non-zero. It must then be multiplication by a non-zero scalar, and so we get
\[ \overline{\tau} = \lambda \rho \text{ for some } \lambda \in \C^*,\]
which shows that $\overline{\tau}$ is a bi-filtered isomorphism, as desired.

The description of the Hodge filtration is an easy computation (keeping in mind the shift by $r$ in the Hodge filtration in \theoremref{thm-main}). For the weight filtration, as $\cB_f$ is pure of weight $n$, we are interested in the monodromy filtration (shifted by $n= \dim X$) on ${\rm Gr}_L^{|L|}(\cB_f)$ with respect to $L(t\de_t)$. 

By \cite{SteenbrinkZucker}*{Rmk. 2.3}, we can write this filtration as
\[ W_{n+\ell} {\rm Gr}_L^{|L|}(\cB_f) = \sum_{j \geq \max\{0,-\ell\}} L(t\de_t)^{j} \ker(L(t\de_t)^{\ell+1+2j}).\]

Note that if $j > 0$, then 
\[L(t\de_t)^{j} \ker(L(t\de_t)^{\ell+1+2i})\subseteq {\rm Im}(L(t\de_t)) \subseteq \sum_{i=1}^r t_i \de_{t_i} {\rm Gr}_V^{|L|}(\cB_f) \subseteq \sum_{i=1}^r t_i {\rm Gr}_L^{|L|-a_i}(\cB_f).\] 

Thus,
\[ W_{n+\ell+r} \cH^r_Z(\cO_X) = \frac{\ker(L(t\de_t)^{\ell+1}) + \sum_{i=1}^r t_i {\rm Gr}_L^{|L|-a_i}(\cB_f)}{\sum_{i=1}^r t_i {\rm Gr}_L^{|L|-a_i}(\cB_f)},\]
which finishes the proof. Note that the shift by $r$ comes from the fact that we are studying the $r$th cohomology in  \theoremref{thm-main}.
\end{proof}

This immediately gives the following:
\begin{cor} For $Z = V(f_1,\dots, f_r) \subseteq X$ a complete intersection of pure codimension $r$ and any non-degenerate slope $L$, we have
\[\widetilde{\alpha}(Z) \geq r+k \iff \de_t^\beta \delta_f \in {}^L V^{|L|}\cB_f \quad \forall\, |\beta| \leq k \iff Z \text{ has } k\text{-Du Bois singularities},\]
\[ \widetilde{\alpha}(Z) > r+k \iff L(t\de_t)\de_t^\beta \delta_f \in {}^L V^{>|L|}\cB_f\quad  \forall \, |\beta| \leq k \iff Z \text{ has } k\text{-rational singularities}.\]
\end{cor}

\subsection{Weighted homogeneous complete intersections}\label{whci} Next, we prove \corollaryref{cor-ComputeMinExp}. We assume that $f_1,\dots, f_r \in \C[x_1,\dots, x_n]$ are weighted homogeneous of degrees $d_1 \leq \dots \leq d_r$ which define $Z \subseteq \A^n_x$, a complete intersection such that $0 \in Z$ is an isolated singular point. Here weighted homogeneous means there exist $w_1,\dots, w_n \in \Z_{>0}$ such that if $\theta_w = \sum_{i=1}^n w_i x_i \de_{x_i}$, then
\[ \theta_w f_j = d_j f_j \text{ for all } j.\]

We assume throughout this section that $d_1 + \dots + d_r \leq |w|$. It was shown in \cite{CDM3}*{Prop. 2.1} that this implies
\begin{equation} \label{eq-GenUpperBound} \widetilde{\alpha}(Z) \leq r + \frac{|w| - \sum_{i=1}^r d_i}{d_r}.\end{equation}

Our goal is the following theorem, which is a strengthening of the result of \cite{CDM3} in the case of Du Bois singularities:
\begin{thm} Let $Z = V(f_1,\dots, f_r) \subseteq X=\A^n_x$ be defined by $f_1,\dots, f_r$ which are weighted homogeneous of degrees $2 \leq d_1\leq \dots \leq d_r$ satisfying $|w| \geq d_1 + \dots + d_r$. Assume $Z$ is a complete intersection with an isolated singularity at $0$. Then
\[ r+ \left\lfloor \frac{|w| - \sum_{i=1}^r d_i}{d_r}\right\rfloor\leq \widetilde{\alpha}(Z)\leq r+ \frac{|w| - \sum_{i=1}^r d_i}{d_r}. \]

Thus, $\lfloor \widetilde{\alpha}(Z)\rfloor = r+ \lfloor \frac{|w| - \sum_{i=1}^r d_i}{d_r}\rfloor$. Moreover, $\widetilde{\alpha}(Z) = \lfloor \widetilde{\alpha}(Z)\rfloor$ if and only if 
\[d_r \,\bigg{|}\, |w| - \sum_{i=1}^r d_i.\]
\end{thm}

To prove the lower bound, we study the ${}^L V$-filtration on $\cB_f = \bigoplus_{\alpha \in \N^r} \cO_X \de_t^\alpha \delta_f$. Here $L = \sum_{i=1}^r d_i s_i$. This is the natural slope to consider in this example, by observing
\[ \theta_w(\delta_f) = L(s) \delta_f,\]
and so
\begin{equation} \label{eq-homogeneity} \theta_w(x^\alpha \de_t^\beta \delta_f) = (L(s-\beta) + w\cdot \alpha)(x^\alpha \de_t^\beta \delta_f),\end{equation}
using the fact that $\theta_w x^\alpha = x^\alpha \theta_w + \theta_w(x^\alpha) = x^\alpha(\theta_w + w\cdot \alpha)$ as elements of $\cD_{\A^n}$, and the fact that $\theta_w$ commutes with $\de_t^\beta$, but $L(s) \de_t^\beta = \de_t^\beta L(s-\beta)$.

We have the following general observation:

\begin{lem} \label{lem-supportGrL} Let $Z = V(f_1,\dots,f_r) \subseteq X$ be a reduced complete intersection of codimension $r$ in a smooth variety $X$. Then for any non-degenerate slope $L = \sum_{i=1}^r a_i s_i$, we have
\[ {\rm Gr}_L^{\lambda}(\cB_f) \text{ is supported on } Z_{\rm sing} \text{ for } \lambda \notin \Z_{\geq |L|},\]
where $|L| = \sum_{i=1}^r a_i$.
\end{lem}
\begin{proof} Let $U \subseteq X$ be an open subset such that $U \cap Z = Z_{\rm reg}$. Then, locally on $U$, the functions $f_1,\dots, f_r$ restrict to part of a system of coordinates on $U$. Write $\de_{f_1},\dots, \de_{f_r}$ for the derivations in this system of coordinates which satisfy
\[ \de_{f_i}(f_j) = \begin{cases} 1 & i =j \\ 0 & i \neq j\end{cases}.\]

In this case, we have that $\cB_f$ is ${}^L V^{0}\cB_f$ coherent, which by \lemmaref{lem-LVtAdic} proves the claim. Indeed, it suffices to show that, for every $\beta \in \N^r$, the element $\de_t^\beta \delta_f \in {}^L V^{0}\cD \cdot \delta_f$. But we have
\[ \de_t^\beta \delta_f = \de_f^\beta \delta_f \in \cD_X \cdot \delta_f\]
over $U$, proving the claim.
\end{proof}

Return now to the case that $Z = V(f_1,\dots, f_r) \subseteq \A^n_x$ is a complete intersection with an isolated singular point at $0$, such that each $f_i$ is weighted homogeneous of degree $d_i$, meaning
\[ \theta_w f_i = d_i f_i.\]

By reordering, we can assume $d_1 \leq \dots \leq d_r$. The following observation is elementary:
\begin{lem} \label{lem-criterionLV} In the situation above, for any $\alpha \in \N^n$ and $\beta \in \N^r$, we have
\[ x^\alpha \de_t^\beta \delta_f \in {}^L V^{\min\{|L|,|w| + w\cdot \alpha- L(\beta)\}}\cB_f,\]
where $|w| = \sum_{i=1}^n w_i$ and $w\cdot \alpha = \sum_{i=1}^n w_i\alpha_i$.
\end{lem}
\begin{proof} Assume $x^\alpha \de_t^\beta \delta_f$ defines a non-zero element of ${\rm Gr}_L^\chi (\cB_f)$ with $\chi < |L|$. Our goal is to establish the inequality $\chi \geq w\cdot \alpha + |w| - L(\beta)$.

As $\chi < |L|$, we have by \lemmaref{lem-supportGrL} above that ${\rm Gr}_L^\chi(\cB_f)$ is supported on $Z_{\rm sing} = \{0\}$. By Kashiwara's lemma (here applied to the study of the eigenspaces of $\theta_w$ on coherent ${\rm Gr}_L^0(\cD_T)$-modules), we can write
\[ {\rm Gr}_L^\chi (\cB_f) = \bigoplus_{\gamma \in \N^n} \cN \de_x^\gamma \delta_0,\]
where $\cN$ is a coherent module over ${\rm Gr}_L^0(\cD_{\A^r_t})$. For any $\eta \in \cN$, we have $x_i(\eta \delta_0) = 0$. Moreover, the Leibniz rule gives
\[ \theta_w \de_x^\gamma = \de_x^\gamma (\theta_w - w\cdot \gamma).\]

Finally, $\theta_w = \sum_{i=1}^n w_i x_i \de_{x_i} = \sum_{i=1}^n w_i \de_{x_i} x_i - |w|$. Putting this together, we see that
\[ \theta_w(\de_x^\gamma(\eta \delta_0)) = \de_x^\gamma(\theta_w \eta \delta_0) = (-|w| - w\cdot \gamma)\de_x^\gamma\eta \delta_0.\]

Hence, in this situation, any generalized eigenvector of $\theta_w$ is actually an eigenvector and its eigenvalue is $\leq -|w|$.

Assume $x^\alpha \de_t^\beta \delta_f$ defines a non-zero element of ${\rm Gr}_L^\chi(\cB_f)$. Then
\[ \theta_w(x^\alpha \de_t^\beta \delta_f) = (L(s-\beta)+w\cdot \alpha)(x^\alpha \de_t^\beta \delta_f),\]
and so
\[ (L(s)+\chi)(x^\alpha \de_t^\beta \delta_f) = (\theta_w + L(\beta) - w\cdot \alpha +\chi)(x^\alpha \de_t^\beta \delta_f).\]

As $L(s)+\chi$ is nilpotent on ${\rm Gr}_L^\chi(\cB_f)$, we see that $x^\alpha \de_t^\beta \delta_f$ is a generalized eigenvector for $\theta_w$ with eigenvalue $w\cdot \alpha - L(\beta) - \chi$. This gives
\[ w\cdot \alpha - L(\beta) - \chi \leq -|w|,\]
and so $\chi \geq w\cdot \alpha - L(\beta) + |w|$, proving the claim.
\end{proof}

\begin{cor} For $k = \lfloor \frac{|w| - |L|}{d_r}\rfloor$, we have
\[ F_{k+r} \cB_f \subseteq {}^L V^{|L|}\cB_f, \quad F_{k+r+1}\cB_f \not \subseteq {}^L V^{|L|}\cB_f.\]
\end{cor}
\begin{proof} To prove this, we show that for any $\beta$ with $|\beta| = k$ we have
\[ \de_t^\beta \delta_f \in {}^L V^{|L|}\cB_f.\]

By the previous lemma, we want to understand when $|w| - L(\beta) \geq |L|$. By varying over all $\beta$ with $|\beta| = k$, the maximal value of $L(\beta)$ is $k d_r$. Hence, for any $p\in \mathbf{Z}$ such that $|w| - |L| \geq p d_r$, we get $F_{p+r} \cB_f \subseteq {}^L V^{|L|}\cB_f$.

The second claim follows from the general upper bound in Equation~\eqref{eq-GenUpperBound}.
\end{proof}

In summary, with the discussion above, we have the following result.

\begin{cor} If $Z = V(f_1,\dots, f_r) \subseteq \A^n_x$ is a complete intersection with isolated singularity at $0$, such that $f_i$ is weighted homogeneous of degree $d_i$ and $d_1 \leq \dots \leq d_r$ and $|w| \geq d_1 + \dots + d_r$, then
\[\lfloor \widetilde{\alpha}_0(Z) \rfloor = r+ \left\lfloor \frac{|w| - |L|}{d_r}\right\rfloor.\]
\end{cor}
\begin{proof} This follows immediately from $F_{k+r} \cB_f \subseteq {}^L V^{|L|}\cB_f$ when $k = \lfloor \frac{|w|-|L|}{d_r}\rfloor$ and the previous proposition.
\end{proof}

Using the general upper bound in Equation~\ref{eq-GenUpperBound}, we have the following:
\begin{cor} If $d_r \big{|}\, |w| - |L|$, we have
\[ \widetilde{\alpha}(Z) = r+ \frac{|w| - |L|}{d_r}.\]
\end{cor}

As we can compute the weight filtration on $\cH^r_Z(\cO_X)$ using the nilpotent operator $L(t\de_t)$ on ${\rm Gr}_L^{|L|}(\cB_f)$, we also have the following:
\begin{lem} If $d_r \nmid |w| - |L|$, then
\[ \widetilde{\alpha}_0(Z) > r+\frac{|w|-|L|}{d_r}.\]
\end{lem}
\begin{proof} Let $k = \lfloor \frac{|w| - |L|}{d_r}\rfloor$. We will show that, under the assumption $d_r \nmid |w| - |L|$, we have
\[ \de_t^{\beta+e_i} \delta_f \in {}^L V^{> |L| - d_i} \cB_f \text{ for all } |\beta| = k, 1\leq i\leq r.\]

Indeed, this implies that $L(t\de_t) F_{k+r} {\rm Gr}_L^{|L|}(\cB_f) = 0$, proving $$F_k W_{n+r} \cH^r_Z(\cO_X) = F_k \cH^r_Z(\cO_X) = P_k \cH^r_Z(\cO_X).$$

Assume there exists $\beta$ with $|\beta| = k$ and $1\leq i \leq r$ such that $$\de_t^{\beta+e_i}\delta_f \in {}^L V^{|L| - d_i}\cB_f \setminus {}^L V^{>|L|-d_i}\cB_f.$$ By \lemmaref{lem-criterionLV}, this means $|L| - d_i \geq |w| - L(\beta+e_i)$, so that $L(\beta) \geq  |w| - |L|$. But $L(\beta) \leq k d_r$, so we get
\[ |w| - |L| \leq k d_r,\]
and so $\frac{|w| - |L|}{d_r} \leq \lfloor \frac{|w| - |L|}{d_r}\rfloor$, contradicting our assumption.
\end{proof}

We end with an explicit computation of the ${}^L V$-filtration in this setting.

\begin{thm} \label{thm-computeLV} Define a filtration
\[ U^\lambda \cB_f = \begin{cases} \sum_{|w| + w\cdot \alpha - L(\beta) \geq \lambda} \cD_X\cdot( x^\alpha \de_t^\beta \delta_f) & \lambda \leq |L| \\ \sum_{i=1}^r t_i U^{\lambda - d_i} \cB_f & \lambda > |L| \end{cases}.\]

In the setting above, we have ${}^L V^\lambda \cB_f = U^\lambda \cB_f$.
\end{thm}
\begin{proof} \lemmaref{lem-criterionLV} shows $U^\bullet \cB_f \subseteq {}^L V^\bullet \cB_f$, so we need only prove the opposite inclusion. It is trivial to check that it satisfies the properties of \propositionref{prop-characterizeLV} except possibly the coherence condition, but this gives the desired containment.
\end{proof}

\bibliography{bib}
\bibliographystyle{abbrv}

\end{document}